\documentclass{article}
\usepackage[utf8]{inputenc}
\usepackage{graphicx,multicol}
\usepackage{amssymb,amsmath,amsthm}
\usepackage[usenames,dvipsnames]{color}
\usepackage{marginnote}
\usepackage{courier}
\usepackage{caption}
\usepackage{listings}
\usepackage{algorithmic}
\usepackage{graphicx}
\usepackage{grffile}
\usepackage{slashbox}
\definecolor{MyPink}{rgb}{0.75686,0.11765,0.95686}
\DeclareCaptionFont{white}{\color{white}}
\DeclareCaptionFont{Blue}{\color{Blue}}

\usepackage{olofmacros}

\graphicspath{{./},{../fig/}}

\DeclareMathOperator{\supp}{supp}

\newtheorem{lem}{Lemma}
\newtheorem{theorem}{Theorem}
\newtheorem{remark}{Remark}

\providecommand{\BigO}[1]{\,\mathcal{O}#1}
\providecommand{\D}[1]{\,\mathrm{d}#1}

\def\veps{v^\varepsilon}

\begin{document}

\title{Analysis of {HMM} for One Dimensional Wave Propagation Problems Over Long Time}
\author{Bj\"orn Engquist\thanks{Department of Mathematics and Institute for Computational Engineering and Sciences, The University of Texas at Austin, 1 University Station C1200, Austin TX 78712, U.S.A (engquist@ices.utexas.edu)}
\and Henrik Holst\thanks{Department of Numerical Analysis, CSC, KTH, 100 44 Stockholm, Sweden (holst@kth.se)}
\and Olof Runborg\thanks{Department of Numerical Analysis, CSC, KTH, 100 44 Stockholm, Sweden (olofr@nada.kth.se) and Swedish e-Science Research Center (SeRC)}
}


\maketitle

\begin{abstract}{}
Multiscale problems are computationally costly to solve by direct simulation because the smallest scales must be represented over a domain determined by the largest scales of the problem.  We have developed and analyzed new numerical methods for multiscale wave propagation following the framework of the heterogeneous multiscale method.  The numerical methods couple simulations on macro- and microscales for problems with rapidly fluctuating material coefficients. The computational complexity of the new method is significantly lower than that of traditional techniques. We focus on HMM approximation applied to long time integration of one-dimensional
wave propagation problems in both periodic and non-periodic medium and show that the dispersive effect that appear after long time is fully captured.
\end{abstract}

\section{Introduction}

There are many application of the wave propagation in heterogeneous media.
Acoustics, elastic or electromagnetic wave propagation in composite material are
typical examples. In seismic exploration an acoustic wave equation must be
solved for highly fluctuating wave velocity over long time with large
disturbances. We consider a representable model problem: The scalar wave
equation in a bounded domain with initial and boundary conditions,
\begin{equation}
	\begin{cases}
		u^{\varepsilon}_{tt} - \nabla \cdot (A^{\varepsilon} \nabla u^{\varepsilon}) = 0, & \Omega \times [0,T^{\varepsilon}], \\
		u^{\varepsilon}(x,0) = f(x), \quad u^{\varepsilon}_t(x,0) = g(x), & \forall x \in \Omega, 
	\end{cases}
	\label{eq:introduction:wave}
\end{equation}
on a smooth domain $\Omega \subset \mathbb{R}^d$ where $A^{\varepsilon}(x)$ is a symmetric, uniformly positive definite matrix. For simplicity we assume that $u^{\varepsilon}$ satisfies periodic boundary conditions.  We assume that $A^{\varepsilon}$ has oscillations on a scale proportional to $\varepsilon \ll 1$.  The solution of \eqref{eq:introduction:wave} will then be a sum of two parts: one coarse scale (macroscale) part, which is independent of $\varepsilon$, and an oscillatory (microscale) part which is highly oscillating in both time and spatial directions on the scale $\varepsilon$.  These kinds of multiscale problems are typically very computationally costly to solve  by traditional numerical techniques. The smallest scale must be well represented over a domain which is determined by the largest scales.  However most often one is only interested in the coarse scale part of the solution.  The goal of our research here is to find an efficient way to compute it.  


Different frameworks for numerical multiscale methods have recently been proposed. The technique we will consider we will follow here is the heterogeneous multiscale method (HMM)~\cite{e2004,e2003,e2007,e2003b}.
These multiscale methods couple simulations on macro- and microscales and compute the coarse scale solution efficiently. In this paper we use HMM for wave propagation and we build on \cite{engquist2011} where we described a HMM multiscale method which captured the coarse scale behavior of \eqref{eq:introduction:wave} for finite time. The main aim here is to show that the same HMM methodology works also for long time, where new macroscale phenomena occurs. 

For finite time and periodic velocity coefficients converges to the solution of the homogenized equations, \cite{bensoussan1978},
\begin{equation}
	\begin{cases}
		\bar{u}_{tt} - \nabla \cdot (\bar{A} \nabla \bar{u}) = 0, & \Omega \times [0,T], \\
		\bar{u}(x,0) = f(x), \quad \bar{u}_t(x,0) = g(x), & \forall x \in \Omega,
	\end{cases}
	\label{eq:introduction:wavebar}
\end{equation}
for fixed $T$ independent of $\varepsilon$, $A^{\varepsilon}(x) = A(x,x/\varepsilon)$ and where $A(x,y)$ is periodic in $y$. The macroscale algorithm in HMM is inspired by \eqref{eq:introduction:wavebar} but it only does not use the explicit form and only uses the flux form, \cite{engquist2011},
\begin{equation}
	\begin{cases}
		u_{tt} - \nabla \cdot F = 0, & \Omega \times [0,T^{\varepsilon}], \\
		u(x,0) = f(x), \quad u_t(x,0) = g(x), & \forall x \in \Omega.
	\end{cases}
\end{equation}
The challenge is to see if the same HMM approximation is robust such that it also is applicable for very long time, $\BigO(\varepsilon^{-2})$, where $\BigO(1)$ dispersive effects appear. These effects are not given by \eqref{eq:introduction:wavebar}. The dispersive property is accurately captured if care is taken to have enough accuracy in each component of HMM. 

The numerical long time solution compares very well with analytic long time effective equation given by Santosa and Symes, \cite{santosa1991}. In the one-dimensional case, when $A^{\varepsilon}(x)=A(x/\varepsilon)$ and $A$ periodic, their long time effective equation will be of the form
\begin{equation}
	u_{tt} - \bar{A} u_{xx} - \beta \varepsilon^2 u_{xxxx} = 0,
	\label{eq:introduction:eff}
\end{equation}
where $\bar{A}$ is the same coefficient as in \eqref{eq:introduction:wavebar} and $\beta$ is a functional of $A$ which can be difficult to determine analytically and numerically. 

The rest of the paper is organized as follows. Section~\ref{section:hmm} contains a description of the numerical algorithm including the added accuracy compared to the technique given in \cite{engquist2011}. Analysis of the method is presented in Section~\ref{section:theory} and the numerical results are contained in the following section. The conclusion are in the final Section~\ref{section:conclusions}.

\section{HMM for the Wave Equation.\label{section:hmm}}

We will now briefly describe how HMM can be applied to \eqref{eq:introduction:wave}.  We assume that there exists two models, a micro model $h(u^{\varepsilon},d^{\varepsilon}) = 0$ describing the full problem, where $u^{\varepsilon}$ is the quantity of interest and $d^{\varepsilon}$ is the problem data (i.e. initial conditions, boundary conditions, \ldots), and a coarse macro model $H(u,d) = 0$, with solution $u$ and data $d$.  The micro model is accurate but is expensive to compute by traditional methods; in our case this model is \eqref{eq:introduction:wave}. The macro model gives a coarse scale solution $u$, assumed to be a good approximation of the microscale solution $u^{\varepsilon}$ and is less expensive to compute.  The model is however incomplete in some sense and requires additional data.  In our case we use $$u_{tt} - \nabla \cdot F = 0$$ with the flux $F$ unknown. This is inspired by the form of the homogenized equation \eqref{eq:introduction:wavebar}.
A key idea in the HMM is to provide the missing data in the macro model using a local solution to the micro model.  Here \eqref{eq:introduction:wave} is solved locally on a small domain with size proportional to $\varepsilon$ and $F$ is given as an average of the resulting microscopic flux $A^{\varepsilon} \nabla u^{\varepsilon}$.  The initial data and boundary conditions ($d^{\varepsilon}$) for this computation is constrained by the macroscale solution $u$. 

It should be noted that even if our numerical methods use ideas from homogenization theory they do not solve any effective (e.g. homogenization or effective medium) equation directly. The goal is to develop computational techniques that can be used when there is no fully known macroscopic PDE.

We will now describe a HMM method for the wave equation \eqref{eq:introduction:wave} which will give useful solutions in two regimes. The first regime is when $T$ is fixed and independent of $\varepsilon$. The other regime is when $T \rightarrow \infty$ as $\varepsilon \rightarrow 0$. We will call this the \emph{long time} regime and the problem itself a long time wave propagation problem.
In this paper we will consider the one dimensional wave propagation problem. Many of the results can be shown to hold in a $d$-dimensional setting. We have in previous works shown higher dimensional results, both theoretical and numerical. 
We demonstrated in the previous papers~\cite{engquist2011,engquist2011c} that our HMM captures the same solution as homogenization (when applicable). In this paper we will primarily investigate how our HMM method handles the long time problem. 
The microscopic variations in the medium introduces dispersive effects in the macroscopic behavior of the solution which becomes notable after long time. Our goal is to show that our HMM method can capture the dispersion with less computational cost than just resolving the full equation.

Let us illustrate the dispersive effects by a numerical example. We consider a one-dimensional example where we solved \eqref{eq:introduction:wave} with $\Omega=[0,1]$, $\varepsilon=0.01$ and,
\begin{equation}
	\left\{
	\begin{aligned}
		& A^{\varepsilon}(x) = A(x/\varepsilon),  \quad A(y) = 1.1 + \sin 2 \pi y, \qquad \forall x \in \Omega, \\
		& u(x,0) = e^{-100 x^2} + e^{-100 (1-x)^2}, \quad u_t(x,0) = 0, \qquad \forall x \in \Omega, \\
	\end{aligned}
	\right.
	\label{eq:ex2:problem}
\end{equation}
for a long time, $T^{\varepsilon}=\mathcal{O}(\varepsilon^{-2})$. Since we have periodic boundary conditions, the solution to the corresponding homogenized equation will be periodic in time with period $1/\sqrt{\bar{A}}=1.47722$. We will compute the solution to \eqref{eq:ex2:problem} for 1, 10 and 100 periods. We see in Fig.~\ref{fig:ex-numerical_illustration} that after 100 periods there is an $\mathcal{O}(1)$ error between the true solution $u^{\varepsilon}$ and the homogenized solution $\bar{u}$ which thus fails to capture the dispersive behavior of the solution after long time.
\begin{figure}[htb]
	\begin{center}
		\includegraphics[width=0.5\textwidth]{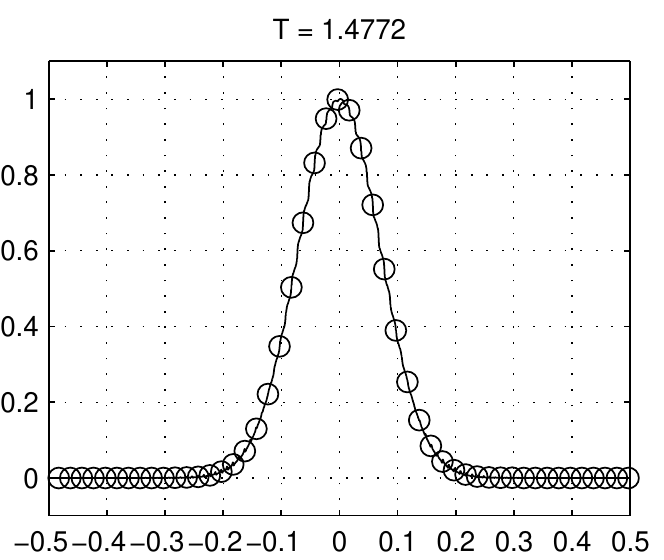}%
		\includegraphics[width=0.5\textwidth]{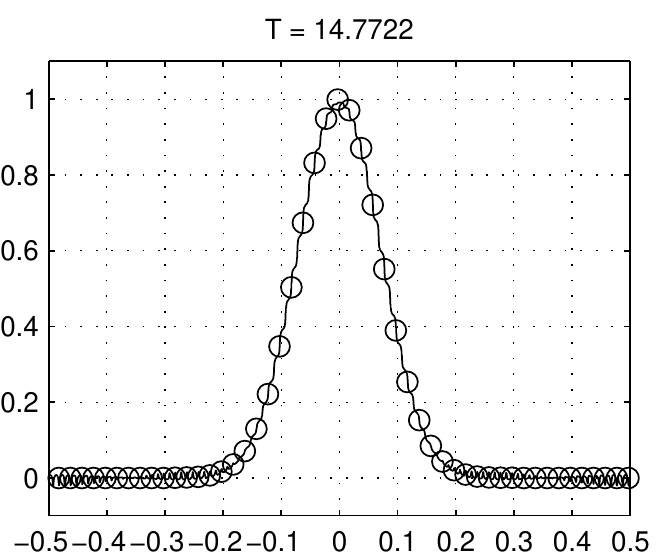}\\
		\includegraphics[width=0.5\textwidth]{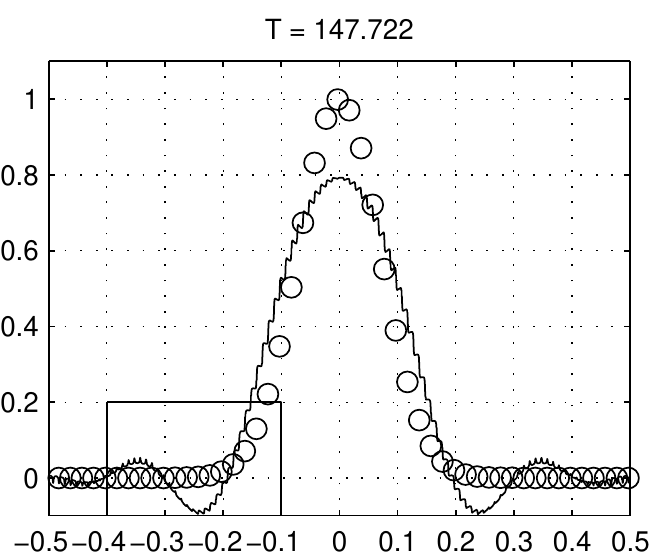}%
		\includegraphics[width=0.5\textwidth]{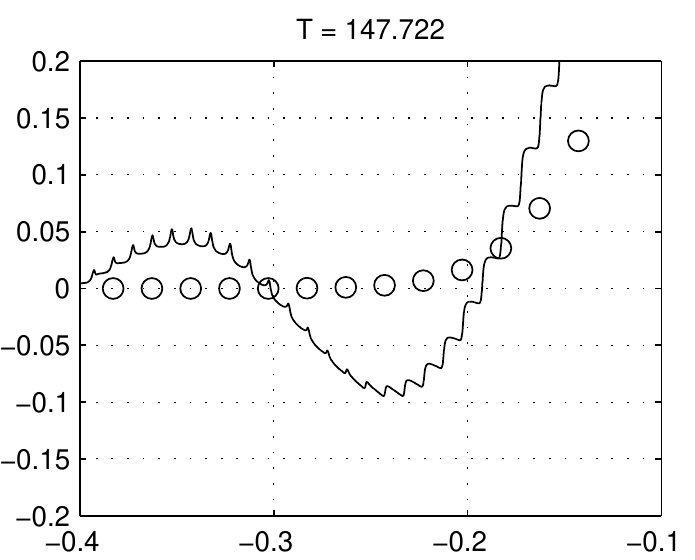}\\
	\end{center}
	\caption{Finite difference computation of \eqref{eq:ex2:problem} at $T=1.47722$, $T=14.7722$ and $T=147.722$ (1, 10 and 100 periods of the homogenized solution) and the corresponding homogenized solution (circles). As we can see the homogenized solution does not capture the dispersive effects that occur.\label{fig:ex-numerical_illustration}}
\end{figure}

\subsection{Description of the Method}

The HMM method we suggest here is described in three separate steps. 
We follow the same strategy as in \cite{abdulle2003} for parabolic equations and in \cite{samaey2006} for the one-dimensional advection equation. See \cite{engquist2007}, \cite{engquist2011} and \cite{abdulle2011} for additional details and proofs.
In step one we give the macroscopic PDE (i.e. the form $H(u,d)=0$) and a corresponding numerical discretization. In step two we describe the microscale problem (microproblem). The initial data for the microproblem is based on local macroscopic data. Finally, in step three we describe how we approximate $F$ from the computed microproblem by taking a weighted average of its solution.

We focus in this paper on the 1D version of \eqref{eq:introduction:wave},
\begin{equation}
	\left\{
	\begin{aligned}
		& u^{\varepsilon}_{tt} - \partial_x (A^{\varepsilon} u^{\varepsilon}_x) = 0, \qquad \Omega \times [0,T^{\varepsilon}], \\
		& u^{\varepsilon}(x,0) = f(x), \quad u^{\varepsilon}_t(x,0) = g(x), \qquad \forall x \in \Omega.
	\end{aligned}
	\right.
	\label{eq:hmm:wave}
\end{equation}
where $\Omega \subset \mathbb{R}$ is a interval of the form $[a,b]$ and $u^{\varepsilon}(x,t)$ is $\Omega$-periodic in $x$.

\begin{remark}
	All integrals are taken over $\mathbb{R}$ unless explicitly stated otherwise.
\end{remark}

\paragraph*{Step 1: Macro model and discretization}

We suppose there exists a macroscale PDE of the form,
\begin{equation} 
	\left\{
	\begin{aligned}
		& u_{tt} - \partial_x F = 0, \qquad Y \times [0,T^{\varepsilon}],  \\
		& u(x,0) = f(x), \quad u_t(x,0) = g(x), \qquad \forall x \in Y, \\
	\end{aligned}
	\right.
	\label{eq:hmm:macroform}
\end{equation}
where $u$ is $Y$-periodic; $F$ is a function of $x$, $u$ and higher derivatives of $u$. We will use this assumption throughout the whole paper. The assumption on \eqref{eq:hmm:macroform} is that $u \approx u^{\varepsilon}$ when $\varepsilon$ is small. In the method we suppose that $F = F(x,u,u_x,u_{xx},u_{xxx})$. In the clean homogenization case we would have $F = \bar{A} u_x$, and in the effective equation \eqref{eq:introduction:eff} $F = \bar{A} u_{x} + \beta \varepsilon^2 u_{xxx}$, but we will not assume knowledge of the exact form of a homogenized equation or any other effective equation.  

We discretize \eqref{eq:hmm:macroform} using central differences with time step $K$ and spatial grid size $H$ in all directions,
\begin{equation}
	U^{n+1}_m = 2 U^n_m - U^{n-1}_m + \frac{K^2}{24 H} \left( -F^n_{m+3/2} + 27 F^n_{m+1/2} - 27 F^n_{m-1/2} + F^n_{m-3/2} \right),
	\label{eq:hmm:macroscheme}
\end{equation}
where $F^n_{m-\frac{1}{2}}$ is $F$ evaluated at $x'=x_{m-\frac{1}{2}}$ and so forth. The scheme is second order in $K$ and fourth order in $H$. We choose to use a fourth order scheme in space since it showed to have better dispersive properties to allow us to avoid some of the numerical dispersion that pollutes the numerical solution.

We will show how $F$ is computed on the grid in the description of the micro problem. The implementation details are described in \cite{holst2011}.
	
\paragraph*{Step 2: Micro problem}

The evaluation of $F$ in each grid point is done by solving a micro problem to fill in the missing data in the macro model. Given the parameters $x'$ and an initial data $Q(x)$, we solve
\begin{equation}
	\left\{
	\begin{aligned}
		& v^{\varepsilon}_{tt} - \partial_x (A^{\varepsilon}(x+x') v^{\varepsilon}_x) = 0, \qquad Y^{\varepsilon} \times [-\tau,\tau], \\
		& v^{\varepsilon}(x,0) = Q(x+x'), \quad v^{\varepsilon}_t(x,0) = 0, \qquad \forall x \in Y^{\varepsilon}, \\
	\end{aligned}
	\right.
	\label{eq:hmm:microproblem}
\end{equation}
where $v^{\varepsilon} - v^{\varepsilon}(x,0)$ is $Y^{\varepsilon}$-periodic and $Q$ depends on the macroscopic state $\{u_m^n\}$, assuming $x'=x_{m-1/2}$, it is typically a third order polynomial interpolating $u^n_{m-2},\ldots,u^n_{m+1}$ in $x_{m-2},\ldots,x_{m+1}$, where $u^n_m \approx U(x_m,t_n)$ in \eqref{eq:hmm:macroform}. Note that we have made a change of variables to center $Y^{\varepsilon}$ around $x'$ via $x-x' \mapsto x$. We keep the sides of the micro box $Y^{\varepsilon}$ of order $\varepsilon$. We will discuss the choice of $Y^{\varepsilon}$ and $Q$ below. 

\paragraph*{Step 3: Reconstruction step}

After we have solved for $v^{\varepsilon}$ for all $Y^{\varepsilon} \times [-\tau,\tau]$ we approximate $F$ by a weighted average of the microscopic flux $A^{\varepsilon}(x+x') v^{\varepsilon}_x$ over $[-\eta,\eta] \times [-\tau,\tau]$ where $[-\eta,\eta] \subset Y^{\varepsilon}$. We choose $\eta$ sufficiently small for information not to have propagated into the region $[-\eta,\eta]$ from the boundary of the micro box $Y^{\varepsilon}$. More precisely, we consider smooth 
averaging kernels $K_\eta$ compactly supported in $[\-\eta,\eta]$, see below in 
Section~\ref{section:kernels}.
We then approximate
\begin{equation}
	F \approx F_{\rm HMM}(x',Q) =\iint K_{\tau}(t) K_{\eta}(x) f^{\varepsilon}(x,t) \D{x} \D{t},
	\label{eq:hmm:Ftildekernel}
\end{equation}
where $f^{\varepsilon}(x,t) = A^{\varepsilon}(x+x') v^{\varepsilon}_x(x,t)$ and $v^{\varepsilon}$ solves \eqref{eq:hmm:microproblem}.

\subsection{Motivation of the Method}

In this section we will give motivation to the steps and procedures of the HMM method for the wave equation, in the previous section. To do this we define the coarse scale solution $u(t,x)$ as a local average in time and space of the fine scale solution,
\be{udef}
u(x,t) = (\mathcal{K} u^{\varepsilon})(x,t),
\ee
where $\mathcal{K}$ is a \emph{local averaging operator} defined as,
\begin{equation}
	(\mathcal{K} v)(x,t) = 
	(K_{\eta} \ast (K_{\tau} \ast v))(x,t) = 
	\iint K_{\eta}(\xi) K_{\tau}(s) v(x-\xi,t-s) \D{\xi} \D{s}.
	\label{eq:hmm:mathcalK}
\end{equation}
The functions $K_\eta$ and $K_\tau$ are smooth, compactly supported, 
averaging kernels described below in Section~\ref{section:kernels}. The operator 
$\mathcal{K}$ commutes with $\partial_t$ and $\partial_x$ on smooth functions, since
\begin{multline}
	\partial_t \mathcal{K} v = 
	\partial_t \iint K_{\eta}(\xi) K_{\tau}(s) v(x-\xi,t-s) \D{\xi} \D{s} \\ =
	\iint K_{\eta}(\xi) K_{\tau}(s) v_t(x-\xi,t-s) \D{\xi} \D{s} =
	\mathcal{K} v_t.
\end{multline}
In the same way we have $\partial_x \mathcal{K} v = \mathcal{K} v_x$.

The coarse scale solution $u$ will then satisfy a PDE of the form $u_{tt} - \partial_x F = 0$, which can be seen as follows. We apply $\mathcal{K}$ to both sides of the wave equation \eqref{eq:introduction:wave},
\begin{equation}
	\mathcal{K} u^{\varepsilon}_{tt} - \mathcal{K} \partial_x (A^{\varepsilon} u^{\varepsilon}_x) = 0.
\end{equation}
By applying the commutation property on $\partial_t \mathcal{K}$ twice and once for $\partial_x \mathcal{K}$, we arrive at
\begin{equation}
	\partial_{tt} (\mathcal{K} u^{\varepsilon}) - \partial_x (\mathcal{K} (A^{\varepsilon} u^{\varepsilon}_x)) = 0,
\end{equation}
which is the PDE for $u$ that we seek,
\begin{equation}
	\left\{
	\begin{aligned}
		& u_{tt} - \partial_x F = 0, \quad F = \mathcal{K} (A^{\varepsilon} u^{\varepsilon}_x),  \\
		& u(x,0) = (\mathcal{K} u^{\varepsilon})(x,0), \quad 
		  u_t(x,0) = (\mathcal{K} u^{\varepsilon}_t)(x,0).
	\end{aligned}
	\right.
	\label{eq:motivation:MacroPDE}
\end{equation}
Note that the flux $F$ is precisely the same type of flux used in Step 3 above.
Hence, if $u^\varepsilon$ is the exact fine scale solution, the corresponding
exact coarse scale solution must satisfy this PDE. To make the HMM procedure
consistent, the same PDE is used for the micro problem. Additionally, the
initial condition should be consistent, 
\be{initconsist}
  u(x,0) = (\mathcal{K} u^{\varepsilon})(x,0).
\ee
In Step 2 of the HMM procedure we are given the desired macroscopic
initial data, i.e. $u(x,0)$. The initial condition $Q(x)$ for the microscopic
solution $u^\varepsilon$ must then be chosen such that \eq{initconsist} holds.
To leading order in $\varepsilon$ one can simply take $Q(x)=u(x,0)$,
but to have higher order accuracy $Q(x)$ must be modified. 
On the other hand, it turns out that the initial data for $u^\varepsilon_t$ does not matter.
We will discuss this {\em consistency} problem below
in \sect{Consistency}.



\subsection{Elements of the Method}

Now we will describe two important factors in the HMM method:
The choice of kernel and how to obtain consistent microscopic initial data for a
given macroscopic data.
We start with a discussion about kernels.

\subsubsection{Kernels} \label{section:kernels}

In this section we introduce smooth compactly supported averaging kernels used
to accurately compute averages of periodic and almost periodic functions.
We will give a short motivation why we choose to work with those
instead of just taking the usual mean value of $\bar{f} = \frac{1}{|\Omega|}\int_{\Omega} f(x) \D{x}$. We will also talk about how some numerical difficulties can be avoided by keeping the kernel on a factorized form.

Suppose $v^{\varepsilon}(x,t)$ is periodic function in $x$ and $t$, with mean value $\bar{v}$, and we wish to compute $\bar{v}$ with a numerical method. If we know the period of $v^{\varepsilon}$ in both $x$ and $t$ we can compute $\bar{v}$ with spectral accuracy by the trapezoidal rule. The situation becomes harder when we do
not know the exact period in one of or both of $x$ and $t$. Suppose we have the values of
$v^{\varepsilon}$ 
available over some domain $D= [-\eta,\eta] \times [-\tau,\tau]$ where both $\eta$ and $\tau$ are larger than but of the same order as $\varepsilon$. The normal mean value
then gives a first order error in $\varepsilon/\eta$ and $\varepsilon/\tau$,
\begin{equation}
	\left| \bar{v} - \frac{1}{\eta \tau} \int_{-\tau}^{\tau} \int_{-\eta}^{\eta} v^{\varepsilon}(\xi,s) \D{\xi} \D{s}  \right| \leq 
	C 
	\left(\frac{\varepsilon}\eta
+\frac{\varepsilon}\tau\right).
	\label{eq:kernel:NoKernelEstimate}
\end{equation}
To improve this error we follow the steps in \cite{engquist2005} and
introduce kernel functions of the following form. First, let $K\in C^q_c(\Real)$
satisfy
\begin{equation}    \lbeq{momentcond}
   \int K(x) x^j dx = \begin{cases} 1, & j=0,\\
   0, & j=1,\ldots, p,
   \end{cases}
   \qquad \supp K \subset [-1,1].
\end{equation}
We let $\mathbb{K}^{p,q}$ denote the space of all such functions.
In the computations we will scale $K$ and write 
$$
  K_\eta(x) = \frac{1}{\eta} K\left(\frac{x}{\eta}\right).
$$
This function inherits the integral 
properties of $K$ in \eq{momentcond}, but $\supp K_\eta\subset[-\eta,\eta]$.
The improvement from normal averaging that one can achieve with $K_\eta$
is shown in the next theorem,
which is a simple adaptation of some results derived in \cite{engquist2005}.
\begin{theorem}\lbtheo{oscapprox}
Let $f(t,s)$ be a 1-periodic function in the second variable whose 
derivative $\partial^r_t f(t,s)$ is bounded for $r=0,\ldots,p$.
Then for $f_\varepsilon(t):=f(t,t/\varepsilon)$ and $K\in {\mathbb K}^{p,q}$
we have
$$
   \left|\int K_\eta(\tau)f_\varepsilon(t+\tau) d\tau-\int_0^1 f(t,s)ds \right| \leq 
   C_1\eta^p+C_2 \left(\frac{\varepsilon}{\eta}\right)^q,
$$
where $C_1$ and $C_2$ depends on the derivatives of $f$ and $K$ but not
on $\eta$ or $\varepsilon$. Moreover, if $f(t,s)=a(t)b(s)$ and $b$ is 1-periodic,
then
\be{second}
   \left|\int K_\eta(\tau)f_\varepsilon(t+\tau) d\tau-\int_0^1 f(t,s)ds \right| \leq 
   C \max_{0\leq r\leq q} ||a^{(r)}||_{L^{\infty}[t-\eta,t+\eta]}\left(\frac{\varepsilon}{\eta}\right)^q.
\ee
\end{theorem}
Hence, the convergence rate in \eqref{eq:kernel:NoKernelEstimate} with respect to $\eta$ and $\tau$ can be improved significantly,
\begin{equation}
	\left| \bar{v} - \iint K_{\eta}(\xi) K_{\tau}(s) v^{\varepsilon}(\xi,s) \D{\xi} \D{s} \right| 
	\leq 
	C 
	\left[\Bigl(\frac{\varepsilon}\eta\Bigr)^q
+\Bigl(\frac{\varepsilon}\tau\Bigr)^q\right].
\end{equation}

In our method we have chosen to use functions in $\mathbb{K}^{p,q}$ which
are polynomials and symmetric, $K(x)=K(-x)$. Since the first $q$ derivatives
of $K$ must vanish at $x=\pm 1$, those polynomials can always be factorized as
\be{Kfactor}
	K(x) = (1-x^2)^{q+1} P(x),
\ee
where $P$ is a degree $p-1$ polynomial. Finding the coefficients for $P(x)$
and subsequently evaluating $K(x)$ using the form \eq{Kfactor}
is much more stable numerically, than directly finding the coefficients for $K$,
in particular for large $p$ and $q$.
This is because the coefficients of $K$ then become very large.
A code which computes the $P$ coefficients for any $p,q$ can be found in \cite{holst2011}.
An example why it is important for the numerical accuracy to factorize the polynomial
is shown in Table~\ref{tab:kernerr}. 

\begin{table}[htb]
\centering
\begin{tabular}{|c|c|c|} \hline
  \backslashbox{$N$}{$K(x)$} & $\overbrace{7475.95\cdots x^{28} + \cdots}^{\text{14 terms}}$ & $(1-x^2)^{q+1} P(x)$ \\ \hline
   10 & 6.0572e-04 & 6.0572e-04 \\
   20 & 1.5465e-07 & 1.5467e-07 \\
   40 & 2.2333e-10 & 2.6219e-10 \\
   80 & 3.6733e-11 & 8.3822e-14 \\
  160 & 4.8612e-11 & 1.1102e-16 \\
  \hline
\end{tabular}
\caption{Here we have plotted the relative error as a function of $N$ (grid intervals) for two ways of evaluating the integral of the same polynomial $K \in \mathbb{K}^{9,9}$. The polynomial evalutions themselves are done with Horner's method (MATLAB builtin function \texttt{polyval}). We clearly see that there is a great benefit in using the factorized $K$ from a numerical perspective. The 28 degree polynomial $K$ has only even powers of $x^i$, but has very large coefficients. The large coefficents and the number of terms makes it difficult to evaulate $K$ accuractly. The expected convergence rate is spectral. \label{tab:kernerr}}
\end{table}

\subsubsection{Consistency of the Microproblem Initial Data}\lbsec{Consistency}

An important aspect in our HMM method is that the initial data for the micro problem is consistent with the current macroscopic state.
In practice we approximate the macroscopic state with a third order polynomial
$U(x)$ which interpolates the macroscopic grid values
$u^n_{m-2},\ldots,u^n_{m+1}$ in $x_{m-2},\ldots,x_{m+1}$.
We should then choose a polynomial $Q(x)$ as initial data
to (\ref{eq:hmm:microproblem}) such that
$$
  U(x) = (\mathcal{K} v^{\varepsilon})(0,x).
$$
To simplify the discussion we introduce the operator $\mathcal{M}$ which
maps the initial data of (\ref{eq:hmm:microproblem}) to $(\mathcal{K} v^{\varepsilon})(0,x)$.
We hence seek $Q$ such that
$$
  U(x) = (\mathcal{M} Q)(x).
$$
We will not be able to find a $Q$ where this is satified exactly. However,
in this section we will show how to find $Q$ given $U$ such this equality
is satisfied to high order in $\varepsilon$.

When the coefficient $A^\varepsilon$ is precisely periodic 
and $Q$ is a third order polynomial,
we show below in \sect{Consistency2} that,
under some simplifying assumptions, 
\be{MQ}
  (\mathcal{M} Q)(x) = Q(x) +\varepsilon^2 \gamma Q''(x) + \BigO(\varepsilon^3),
\ee
where $\gamma$ is a constant. Hence, if we take $Q=U$ we make a
$O(\varepsilon^2)$ error. For the long time problem we will need better accuracy.

We assume that the initial data is of the form
$Q(x) = c_0 + c_1 x + c_2 x^2 + c_3 x^3$ and
make that ansatz that
$$
  (\mathcal{M} Q)(x) \simeq U(x):=\tilde{c}_0 + \tilde{c}_1 x + \tilde{c}_2 x^2 + \tilde{c}_3 x^3.
$$
Let $\cvec^T = \begin{bmatrix} c_0 & c_1 & c_2 & c_3 \end{bmatrix}$, $\tilde{\cvec}^T = \begin{bmatrix} \tilde{c}_0 & \tilde{c}_1 & \tilde{c}_2 & \tilde{c}_3 \end{bmatrix}$ and $\x^T = \begin{bmatrix} 1 & x & x^2 & x^3 \end{bmatrix}$. 
Then $Q(x)=\x^T\cvec$ and $U(x)=\x^T\tilde{\cvec}$.
We define $\e_i$ as the unit vector with a one in the $i$th component (counting from zero) and $I$ as the $4 \times 4$ identity matrix.
For consistency we want to figure out which $\cvec$ that gives us a particular $\tilde{\cvec}$ and then 
compute the flux $F$ for that $\tilde{\cvec}$. 
As an example, assume that $Q(x) = x^3$, which would mean that $\cvec=\e_3$. We cannot use this as initial data to obtain the correct flux as it will be inconsistent (to order $\varepsilon^2$), 
as seen from \eq{MQ} and as
described briefly in \cite{engquist2011c}. 
Instead we wish to figure out which initial data $\x^T\cvec$ that gives us a $\x^T \tilde{\cvec}$ such that $\tilde{\cvec} = \e_3$ and use that $\x^T \cvec$ as initial data. We will show that it is not necessary to recompute any micro problems once we have solved just four independent micro problems.

Since $\mathcal{M}$ is a linear operator we can represent it by a matrix $M$ acting between
the coefficients $\cvec$ and $\tilde{\cvec}$,
\be{Mceq}
  \tilde{\cvec} = M \cvec.
\ee
We now note that we can find the elements of $M$ from solving four separate microproblems
as follows. For $i=0,\ldots, 3$ we take
$Q_i(x) = \x^T\e_i = x^i$. Solving the corresponding
microproblem gives $U_i:=\mathcal{M}Q_i$, which we can evaluate
in four points. We fit a third order polynomial to these points and take this
as an approximation of $U_i(x)=:\x^T\tilde{\cvec}_i$. Then,
$\tilde{\cvec}_i$ is the $i$-th column of $M$,
$$
  \tilde{\cvec}_i = M \e_i = 
  \begin{bmatrix}
		m_{0,i} \\
		m_{1,i} \\
		m_{2,i} \\
		m_{3,i} \\
	\end{bmatrix}.
$$
Once we have found $M$ we can solve \eq{Mceq} to obtain $\cvec$ from $\tilde{\cvec}$.
We call $M$ the \emph{correction matrix}.
Based on the scaling in \eq{MQ} we can see that $M$ is a
$\varepsilon^2$-perturbation of the identity, $M=I+\varepsilon^2 M_0$
for some $M_0$.

Once we have computed the $M$ matrix we can also
compute the correct flux,
without solving more micro problems.
For each of the four problems solved we record the corresponding
flux and denote it by $f_i$. Moreover, we
set $\f^T = \begin{bmatrix} f_0 & f_1 & f_2 & f_3 \end{bmatrix}$.
Since the flux computation is also linear in the input $\cvec$ the computed
flux from $Q(x)$ is simply
$$
   F(Q) = \f^T\cvec.
$$
%
Thus, given $U=\x^T \tilde{\cvec}$ the
corrected flux is
\begin{equation}
  \tilde{F}(U) = \tilde{\f}^T \tilde{\cvec}, \qquad \tilde{\f} := M^{-T} \f.
  \label{eq:tildef}
\end{equation}



Let $N$ be the number of grid points on the shifted grid $x_{m-\frac{1}{2}}$ ($1 \leq m \leq N$).
We determine the coefficients in the $4 \times 4$ matrices correction matrices $M^{(m)}$ with the following algorithm
\begin{algorithmic}
\FOR{$m = 1 \to N$}
	\STATE{$x' \gets x_{m-\frac{1}{2}}$}
	\FOR{$i = 1 \to 4$}
		\STATE{Initial data $Q(x) \gets x^{i-1}$, i.e. $c_i=1$ and other $c_j=0$}
		\STATE{Solve micro problem \eqref{eq:hmm:microproblem} $v^{\varepsilon}$ for all $(x,t) \in [-\eta,\eta] \times [-\tau,\tau]$}
		\STATE{Compute $Y_k \gets \mathcal{K} (u^{\varepsilon})(\Delta_k,0)$ for some distinct $\Delta_k$ around $x=0$}
		\STATE{Find a third degree polynomial that fits $(\Delta_k,Y_k)$, which gives $\tilde{c}$.}
		\STATE{$M^{(m)}_{ij} = \tilde{c}_{j-1}, \quad 1 \leq j \leq 4$}
	\ENDFOR
\ENDFOR
\label{alg:Amatrix}
\end{algorithmic}

So far we have only been concerned with the initial data for $\veps$.
In (\ref{eq:hmm:microproblem}) we also need initial data for $\veps_t$.
However, as the following derivation shows, the $\veps_t$ data does
not affect the flux when we use symmetric kernels.
We consider the wave propagation problem of the form
\begin{equation}
	\left\{
	\begin{aligned}
		& v^{\varepsilon}_{tt} - \partial_x (A^{\varepsilon} v^{\varepsilon}_x) = 0, \\
		& v^{\varepsilon}(x,0) = f(x), \quad v^{\varepsilon}_t(x,0) = g(x).
	\end{aligned}
	\right.
	\label{eq:perturbedwave}
\end{equation}
We can split the solution
into two parts $v^{\varepsilon} = v_1^{\varepsilon} + v_2^{\varepsilon}$ where $v_1^{\varepsilon}$ solves \eqref{eq:perturbedwave} with the initial conditions
\begin{equation*}
	v_1^{\varepsilon}(x) = f(x), \quad \partial_t v_1^{\varepsilon}(x,0) = 0,
\end{equation*}
and $v_2^{\varepsilon}$ solves \eqref{eq:perturbedwave} with the initial conditions
\begin{equation*}
	v_2^{\varepsilon}(x) = 0, \quad \partial_t v_2^{\varepsilon}(x,0) = g(x).
\end{equation*}
The solution $v_2^{\varepsilon}$ will be anti-symmetric in time, i.e. $v_2^{\varepsilon}(-t,x) = -v_2^{\varepsilon}(t,x)$, since $A^\varepsilon$ is time-independent. The same
holds for $A^{\varepsilon} \partial_x v_2^{\varepsilon}$.
Then, because our kernels are symmetric, $\mathcal{K}(A^{\varepsilon} \partial_x v_2^{\varepsilon})(0,x)=0$ and
$$
F(0,0) = \mathcal{K}(A^{\varepsilon} \partial_x v^{\varepsilon})(0,0)
=\mathcal{K}(A^{\varepsilon} \partial_x v_1^{\varepsilon})(0,0),
$$
independent of $g(x)$. In the computations we therefore simply take $g\equiv 0$.

\subsection{Changes of the HMM Method for Finite Time for Accurate Long Time Wave Propagation} \label{section:changes}

We have have made three important changes to the HMM method for finite time~\cite{engquist2011} to capture the effective solution for long time wave propagation:
\begin{enumerate}
\item The form of the macroscopic PDE (\ref{eq:introduction:eff})
is the same, $u_{tt}-\partial_x F=0$, but the flux $F=\bar{A}u_x + \beta \varepsilon^2 u_{xxx}$ contains a third derivative of the macroscopic solution. To capture the additional information we need to pass more information to the micro problem. The initial data is chosen as a third order polynomial, instead of a linear polynomial as for the HMM method for finite time.
\item We need to make sure that the initial data in the micro problem is consistent to a very high accuracy with the macroscopic solution.
As we described in \sect{Consistency}, the initial data to the micro problem needs to be consistent with the macroscopic data to accurately capture the dispersive effects which are important for long time wave propagation.
\item Since we need to accurately represent also the second term in the flux, $\beta \varepsilon^2 u_{xxx}$, the error in the flux computation \eqref{eq:hmm:Ftildekernel} must be smaller than $\BigO(\varepsilon^2)$. By \theo{oscapprox}, the error $\mathcal{E}$ in the flux computation is proportional to,
$$
	\mathcal{E} \sim \eta^{p} + \tau^{p} + (\varepsilon/\tau)^q + (\varepsilon/\eta)^q.
$$
Thus, we need to choose $p$, $q$, $\eta$ and $\eta$ such that $\mathcal{E} \ll \varepsilon^2$, asymptotically. 
This means that the micro box must be larger than $\varepsilon$.
For instance, if $p=q>2$ we can take take $\alpha<1-2/q$. Then
$$
\eta,\tau \sim \varepsilon^{\alpha} 
\quad\Rightarrow\quad \mathcal{E} = O\left(\varepsilon^{p-2}
+\varepsilon^{(1-\alpha)q}\right)\ll O(\varepsilon^2).
$$
Recall that in the finite time case we always take $\eta,\tau\sim\varepsilon$, 
This hence explains why we need to have more accurate kernels or bigger micro boxes in the long time case. In order to maintain a low computational cost we usually want small micro box, which can be obtained by using a very regular kernel, i.e. large 
$q$.
\end{enumerate}



%

\begin{remark}
We use the same numerical scheme \eqref{eq:hmm:macroscheme} to solve the micro problem \eqref{eq:hmm:microproblem},
where $F^n_{m-\frac{1}{2}}$ is computed as,
\begin{equation}
	F^n_{m-\frac{1}{2}} = \frac{A^{\varepsilon}(x_{m-\frac{1}{2}}+x') }{24 h} 
	\left[  -v^n_{m+3/2} + 27 v^n_{m+1/2} - 27 v^n_{m-1/2} + v^n_{m-3/2} \right].
\end{equation}
\end{remark}

\section{Theory} \label{section:theory}

In this section we will analyze the flux used
in HMM. In particular we want to compare it with the corresponding
flux for the effective equation. We will focus on the case where we
have a good understanding of this equation, namely when the
coefficient $A^\varepsilon(x)$ is precisely periodic.
We thus consider the micro problem centered at $x=0$
extended to the whole space,
\be{waveeqT}
   \left\{
   \begin{aligned}
     & \veps_{tt} = \nabla A(x/\varepsilon)\nabla\veps, \\
     & \veps(0,x)=Q(x),\quad \veps_t(0,x)=0,\qquad x\in\Real^d,
   \end{aligned}
   \right.
\ee
where $A(y)$ is 1-periodic. 
In \cite{engquist2011} we studied the finite time case where $Q(x)$ was a
first order polynomial. 
We could then prove that
\begin{equation}
	F_{\rm HMM}(Q,\alpha) = F_{\rm HOM}(Q) + \mathcal{O}\left(\alpha^{q}\right),
	\qquad F_{\rm HOM}(Q) := \bar{A}\grad Q,
	\label{eq:FHMM_Q_alpha}
\end{equation}
where we have defined
$$
   \alpha = \frac{\varepsilon}{\eta}.
$$
This means that when the microbox size $\eta$ becomes large in comparison
to $\varepsilon$, hence when $\alpha\to 0$, the HMM procedure
 generates results close to a direct discretization of the homogenized equation \eqref{eq:introduction:wavebar}. 
 
 The restriction to linear initial data made the analysis significantly easier,
 since then $\veps(t,x)-\veps(0,x)$ is periodic for all $t$ and the solution
 could be analyzed via a rather straightforward spectral decomposition. For the long
 time case we need to use initial data $Q(x)$ which is 
 at least a cubic polynomial. To analyze this case we need to use the
 full Bloch theory for wave propagation in a periodic medium.
 This was also used by Santosa and Symes in 
\cite{santosa1991} to describe the effective properties for periodic
medium over long time. We start by reviewing this theory. 
 
\subsection{Long Time Homogenization}

The theory for long time homogenization formally
 extends the validity of the effective model \eqref{eq:introduction:wavebar} 
 from $T=\mathcal{O}(1)$ up to time $T=\mathcal{O}(\varepsilon^{-2})$.
 It does this by adding more terms to the effective equation.
 The basic model was 
 presented in \cite{santosa1991}.

We start with some notation and definitions for Bloch waves.
We let $\omega_m^2$ and $\psi_m$ be the eigenvalues and eigenfunctions of the shifted cell (eigenvalue) problem~\cite[pp.~614]{bensoussan1978},
\begin{equation*}
	\left\{
	\begin{aligned}
		& -\left(\nabla_y + i k \right) A(y) \left(\nabla_y + i k\right) \psi(y,k) = \omega^2(k) \psi(y,k), \qquad Y \times Y , \\
		& \text{$\psi(y,k)$ is $Y$-periodic in $y$}, \\
	\end{aligned}
	\right.
\end{equation*}
where $Y=[0,1]^d$ and $k \in \mathbb{R}^d$. Let $\phi_m(x,k)$ be the 
$\varepsilon$-scaled Bloch-waves,
\begin{equation*}
	\phi_m(x,k) = \psi_m(x/\varepsilon,\varepsilon k) \exp(ikx),
\end{equation*}
which satisfies
\begin{equation*}
	- \nabla_x \left( A\left(\frac{x}{\varepsilon}\right) \nabla_x \phi_m \right) = \frac{1}{\varepsilon^2} \omega^2_m(\varepsilon k) \phi_m.
\end{equation*}

We will now consider the general solution to \eq{waveeqT}
when we assume that the initial data $Q(x)$
is band limited with smallest wavelength larger than $\varepsilon$,
\be{compactsupp}
    Q(x) = \int \hat{Q}(k) e^{ixk}\D{k},\qquad \supp\;\hat{Q}\subset[-\pi/\varepsilon,\pi/\varepsilon].
\ee
We let $V_m$ and $\hat{q}_m$ be the parts of $\veps$ and $Q$ 
that project on $\phi_m$, 
\begin{equation*}
	V_m(t,k) = \int \veps(t,x) \phi^{\ast}_m(x,k) \D{x},\qquad
	\hat{q}_m(k) = \int Q(x) \phi^{\ast}_m(x,k) \D{x}.
\end{equation*}
Then
$$
		\veps(t,x) = \sum_{m=0}^{\infty} \int_{Y/\varepsilon} V_m(t,k) \phi_m(x,k) \D{k}.
$$
Upon inserting this in \eq{waveeqT} we see that $V_m$ will satisfy the ODE
$$
 \partial_{tt} V_m + \frac{1}{\varepsilon^2} \omega^2_m(\varepsilon k) V_m =0,
 \qquad V_m(0,k)=\hat{q}_m(k), \quad \partial_{t} V_m(0,k)=0.
$$
By solving this we
can then write down the exact solution to \eq{waveeqT} as
\begin{align}\lbeq{vexact}
  \veps(t,x) &= 
  \sum_{m=0}^\infty
  \int_{Y/\varepsilon}
  \hat{q}_m(k)
    \phi_m(x,k)\cos(\omega_m(\varepsilon k)t/\varepsilon)
\D k =: \sum_{m=0}^\infty v_m(t,x),
\end{align}
where
$$
  v_m(t,x) = 
  \int_{Y/\varepsilon}
  \hat{q}_m(k)
    \phi_m(x,k)\cos(\omega_m(\varepsilon k)t/\varepsilon)
\D k.
$$
The following theorem from \cite{santosa1991} 
states that $v_0$ is the leading order contribution to $\veps$. In
fact the sum of the terms with $m \geq 1$ is bounded by $\mathcal{O}(\varepsilon)$
in $L^2$, uniformly in time.
\begin{theorem}[Santosa \& Symes]
	Suppose $\veps$ solves \eq{waveeqT} with $d=3$ and initial $Q$ satisfying
	\eq{compactsupp}. Then
	\begin{equation*}
		\int_{\mathbb{R}^3} \left|\sum_{m=1}^\infty v_m(t,x) \right|^2 \D{x} \leq C \varepsilon^2.
	\end{equation*}
	Here $C$ is independent of $\varepsilon$ and $t$ but depends on the $H^2$-norm of the initial data $f$ and the $L^{\infty}$-norm of $A$ and $\partial A$. 
	\label{thm:santosa1991}
\end{theorem}
See \cite{santosa1991} for proof. The long time homogenized
equation is an equation, approximately, satisfied by 
the leading term $v_0$ derived in \cite{santosa1991}. 
In one dimension it is of the form
\begin{equation}
	\bar{u}_{tt} = \bar{A} \bar{u}_{xx} + \varepsilon^2\beta\bar{u}_{xxxx}.
	\label{eq:theory:effective}
\end{equation}
The homogenized coefficients $\bar{A}$ and $\beta$ are given as derivatives
of 
$$
   \Omega(k) := \omega_0(k)^2,
$$
not to be confused with the spatial domain $\Omega$, evaluated at $k=0$,
\be{abarbeta}
	\bar{A} = \frac{1}{2!} \Omega''(0), \qquad
	\beta = -\frac{1}{4!} \Omega''''(0).
\ee
This model is in principle valid up to time $t={O}(\varepsilon^{-2})$.  
However, the equation is not
well-posed since the sign of $\beta$ is wrong. When the grid is refined
the numerical solution becomes unstable.
Still, solving this numerically on a coarse grid gives solutions that
agrees well with the long time behaviour of \eq{waveeqT} as demonstrated
in \cite{santosa1991,engquist2011c}. 
See \cite{engquist2011c} for further discussion and analysis of this issue.

\begin{remark}
There is another formulation of the effective equation 
which is also valid up to $T=\mathcal{O}(\varepsilon^{-2})$. 
This equation is well-posed but the higher order correction term 
involves time derivatives. It is of the form
$$
	\bar{u}_{tt} = \bar{A} \bar{u}_{xx} + \varepsilon^2\frac{\beta}{\bar{A}}\bar{u}_{xxtt}.
$$
It was recently analyzed in \cite{lamacz2011} where convergence was established rigorously
also for the long time case.
\end{remark}

\subsection{The Flux}

As seen in \eqref{eq:theory:effective} and \eq{abarbeta}
the homogenized flux is
\be{Fhom}
  F_{\rm HOM}(Q) = \bar{A}Q' + \beta \varepsilon^2 Q'''
  = \frac{1}{2!} \Omega''(0)Q' -\varepsilon^2 \frac{1}{4!} \Omega''''(0)Q'''.
\ee
We will here analyze the expression for the corresponding HMM flux 
in the one-dimensional case.
This flux is computed using the averaging kernels
from Section \ref{section:kernels}. Let $\veps$ be a solution to the HMM micro problem \eq{waveeqT}. Henceforth we drop
the explicit dependence on initial data $Q$ in the notation for $F_{\rm HMM}(Q,\alpha)$.
Then,
\begin{align}\lbeq{Ftilde}
  F_{\rm HMM}(\alpha):={}&\int\int K_\eta(x)K_\eta(t)a(x/\varepsilon)\partial_x \veps(t,x)\D{x}\D{t}
\nonumber\\
  ={}&\sum_{m=0}^\infty  \int\int K_\eta(x)K_\eta(t)a(x/\varepsilon)  \times \nonumber\\
  {}&\int_{-1/2\varepsilon}^{1/2\varepsilon}
  \hat{q}_m(k)\partial_x
    \psi_m\left(\frac{x}{\varepsilon},\varepsilon k\right)e^{ikx}
 \cos\left(\frac{\omega_m(\varepsilon k)t}{\varepsilon}\right)
\D{k}\D{x}\D{t}.
\end{align}
The goal is to study how $F_{\rm HMM}(\alpha)$ behaves as $\alpha\to 0$, 
hence, when the microbox size
$\eta$ becomes large in comparison to $\varepsilon$. (We always assume $\alpha\leq 1$.)
We first note that for any $r(x)$, the expression
\begin{align*}
   \int K_\eta(x) r(x/\varepsilon)\D{x} &=  \int \frac{1}{\eta}K(x/\eta) r(x/\varepsilon)\D{x} = \frac{\varepsilon}{\eta}
     \int K(x\varepsilon/\eta) r(x)\D{x} \\
     &= \alpha
     \int K(x\alpha) r(x)\D{x},
\end{align*}
is a function of the ratio $\alpha=\varepsilon/\eta$, not of $\eta$ or $\varepsilon$
individually.
We therefore set
$$
   b_m(k) = \int_{-1/2}^{1/2}\psi_m^*(x, k)\D{x},
   \qquad
 s_m(k,\alpha) 
   = \int K_\eta(t)
   \cos(\omega_m(k)t/\varepsilon)
   \D{t}.
$$
Moreover, let
$$
  w_m(k,\alpha) = 
   \varepsilon \int K_\eta(x)a(x/\varepsilon)\partial_x \psi_m(x/\varepsilon,k)e^{ikx/\varepsilon} \D{x},
$$
and finally define
\be{Wdef}
   W(k,\alpha) = \sum_{m=0}^\infty b_m(k)s_m(k,\alpha)w_m(k,\alpha).
\ee
We can then prove the following theorem which, 
under a smoothness assumption on $W$, shows that when initial data
is a polynomial 
the HMM flux (scaled by $\varepsilon$) is always a polynomial
in $\varepsilon\partial_x$ of the same order, whose coefficients only
depend on $\alpha$, not directly on $\varepsilon$. 
\begin{theorem}\lbtheo{FW}
Suppose $a\in C^1$, $W(\cdot,\alpha)\in L^{2}(-\pi,\pi)$
and \eq{compactsupp} holds.
Then
$$
  F_{\rm HMM}(\alpha)=
  \frac{1}{\varepsilon}
  \int_{\Real} \hat{Q}(k)W(\varepsilon k,\alpha)\D{k}.
$$
Moreover, if $Q$ is a polynomial of degree $r$ and $W(\cdot,\alpha)\in C^r(-\pi,\pi)$,
$$
  F_{\rm HMM}=
 \varepsilon^{-1}a_0(\alpha)Q(0) 
 +a_1(\alpha) Q'(0)
 +\varepsilon a_2(\alpha)Q''(0) + \cdots +
 \varepsilon^{r-1}a_r(\alpha)\partial_x^{r}Q(0).
$$
where
\be{acoeff}
   a_r(\alpha) = \frac{i^{-r}}{r!}\partial_k^{r}W(0,\alpha).
\ee
\end{theorem}
\begin{proof}
We prove this theorem in several steps, beginning with a lemma.

\begin{lem}\lblem{Fflemma}
Suppose the support of $\hat{Q}$ lies in $[-\pi/\varepsilon,\pi/\varepsilon]$. 
Then 
$$
\hat{q}_m(k) = \hat{Q}(k)b_m(\varepsilon k), 
$$
and in particular
$\supp\;\hat{q}_m\subset \supp\; \hat{Q}$.
\end{lem}
\begin{proof}
Since $\psi_m(x,k)$ is 1-periodic in $x$ and belongs to $L^2$ for each fixed $k$
we can expand it in a Fourier series,
$$
  \psi_m(x,k)= \sum_n b^*_{m,n}(k)e^{-i2\pi nx},
  \qquad b_{m,n}(k) = \int_{-1/2}^{1/2}\psi_m^*(x,k)e^{-i2\pi n x}\D{x}.
$$
Then
\begin{align*}
   \hat{q}_m(k) &= \int Q(x)\psi_m^*(x/\varepsilon,\varepsilon k)e^{-ikx}\D{x}\\
   &=\sum_nb_{m,n}(\varepsilon k)\int Q(x)e^{-i(k-2\pi n/\varepsilon)x}\D{x}
   =\sum_nb_{m,n}(\varepsilon k)\hat{Q}(k-2\pi n/\varepsilon)).
\end{align*}
But the support of $\hat{Q}$ is in $[-\pi/\varepsilon,\pi/\varepsilon]$ and when $n\neq 0$,
$$
\left| k-\frac{2\pi n}{\varepsilon}\right|\geq -|k|+\frac{2\pi}{\varepsilon}
\geq -\frac{\pi}{\varepsilon}+\frac{2\pi}{\varepsilon}
= \frac{\pi}{\varepsilon},
$$
so the only contribution to the sum is when $n=0$. Hence,
$$
\hat{q}_m(k) = b_{m,0}(\varepsilon k)\hat{Q}(k),
$$
and the result follows.
\end{proof}
From the result of this lemma we then get
\begin{align*}
  F_{\rm HMM}(\alpha)=&{}
  \int_{-\pi/\varepsilon}^{\pi/\varepsilon}
  \sum_{m=0}^\infty  \int\int K_\eta(x)K_\eta(t)a(x/\varepsilon)    
  \hat{q}_m(k)\times\\
 &\partial_x
    \psi_m(x/\varepsilon,\varepsilon k)
 \exp(ikx)\cos(\omega_m(\varepsilon k)t/\varepsilon)
\D{x}\D{t}\D{k}
  \\
=&{}  
  \int_{-\pi/\varepsilon}^{\pi/\varepsilon}\hat{Q}(k)
  \sum_{m=0}^\infty  b_m(\varepsilon k)s_m(\varepsilon k,\alpha)\times\\
  &\int K_\eta(x)a(x/\varepsilon)\partial_x
    \psi_m(x/\varepsilon,\varepsilon k)
 \exp(i\varepsilon kx/\varepsilon)\D{x}\D{k}
  \\
=&{}  
  \frac{1}{\varepsilon}\int_{\Real}\hat{Q}(k)
  \sum_{m=0}^\infty  b_m(\varepsilon k)s_m(\varepsilon k,\alpha)
  w_m(\varepsilon k,\alpha)\D{k}.
\end{align*}
This gives the result in the first part of the theorem.
To show the second part of the theorem, we note that since
 the mapping $Q\mapsto F_{\rm HMM}$ is linear, 
we only need to prove the result for $Q(x)=x^r$.
Polynomials are not in $L^2$ so we first regularize $Q$.
Let $M_\varepsilon\in C_c^\infty(\Real)$ have the properties that
$M_\varepsilon(k)=0$ for $|k|>\pi/\varepsilon$, $M_\varepsilon(0)=1$ and
$M_\varepsilon^{(\ell)}(0)=0$ for $\ell=1,\ldots, r$.
Then define
$$
  \tilde{Q}_a(x) = \iint M_\varepsilon(k)\chi_{I_{a}}(y)y^re^{ik(x-y)}\D{y}\D{k},\qquad
 I_a:=[-a,a].
$$
By construction this function is in $L^2$ and band limited to $[-\pi/\varepsilon,\pi/\varepsilon]$ for all values of $a$. We let $\hat{Q}_a$ denote its Fourier transform
and $\tilde{F}_a$ the flux corresponding to initial data $\tilde{Q}_a$. Then
$$
   \hat{Q}_a(k) = i^r M_\varepsilon(k)\frac{\D^r}{\D{k}^r} D_a(k),\qquad
   D_a(k) = \int_{-a}^ae^{-iky}\D{y}.
$$
and
\begin{align*}
  \tilde{F}_a&=
   \frac{1}{\varepsilon}
  \int \hat{Q}_a(k)W(\varepsilon k,\alpha)\D{k}=  
  \frac{i^r}{\varepsilon}
  \int
  D_a^{(r)}(k) M_\varepsilon(k)W(\varepsilon k,\alpha)\D{k}\\
  &=  
 \frac{i^r}{\varepsilon}
  \int D_a(k) \frac{\D^r}{\D{k}^r}
  M_\varepsilon(k)W^{(r)}(\varepsilon k,\alpha)\D{k}.
\end{align*}
This is valid for all $a$. We can then take $a\to\infty$. Since
$D_a(k)\rightharpoonup \delta(k)$ and the derivatives of $M_\varepsilon$ vanishes
at $k=0$,
$$
\lim_{a\to\infty}
  \tilde{F}_a=
\left. \frac{i^r}{\varepsilon}
  \frac{\D^r}{\D{k}^r}
  M_\varepsilon(k)W^{(r)}(\varepsilon k,\alpha)\right|_{k=0}
 =\frac{(i\varepsilon)^r}{\varepsilon}W^{(r)}(0,\alpha).
$$
Moreover,
\begin{align*}
   \lim_{a\to\infty} \tilde{Q}_a(x)&=
   \lim_{a\to\infty}\int \hat{Q}_a(k) e^{ikx}\D{k}
  = i^r\lim_{a\to\infty} \int M_\varepsilon(k)D_a^{(r)}(k)
   e^{ikx}\D{k}\\
  &= (-i)^r\lim_{a\to\infty} \int D_a(k) \frac{\D^r}{\D{k}^r} M_\varepsilon(k)
   e^{ikx}\D{k}=
   (-i)^r\left.\frac{\D^r}{\D{k}^r} M_\varepsilon(k)
   e^{ikx}\right|_{k=0}\\
   &= (-i)^r(ix)^r = x^r.
\end{align*}
This proves the theorem.
\end{proof}

\subsubsection{Properties of $W(k,\alpha)$}

In this section we prove a theorem about
$W(k,\alpha)$.
What we would like to show is that the coefficients \eq{acoeff}
in \theo{FW} satisfy $a_0(\alpha)=0$ and
\be{ahyp}
   a_r(\alpha) = \frac{i^{-r}}{r!}\partial_k^{r}W(0,\alpha) = \frac{i^{-r+1}}{(r+1)!}\left.\frac{\D^{r+1}}{\D{k}^{r+1}}\Omega(k)B(k)\right|_{k=0}
   +O(\varepsilon^{-r+1}\alpha^q), \qquad r\geq 1,
\ee
where we have defined
$$
  B(k):=|b_0(k)|^2.
$$
In other words, we would have
$$
\partial_k^{r}W(0,\alpha) = \frac{i}{r+1}\left.\frac{\D^{r+1}}{\D{k}^{r+1}}\Omega(k)B(k)\right|_{k=0}
   +O(\varepsilon^{-r+1}\alpha^q), \qquad r\geq 1.
$$
This would be consistent with our numerical experiments
and with \eq{Fhom} for $r\leq 2$ when $\alpha\to 0$. 
We can see this in the following way.
Recalling that $\omega_0(0)=0$ we note that $\Omega(0)$
and $\Omega^{(p)}(0)=0$ for $p$ odd. Moreover, from \lemlab{blem} below we
also have $B'(0)=0$ and $B(0)=1$. Hence, 
$$
  a_1(\alpha) = \frac12(\Omega''(0)B(0)+2\Omega'(0)B'(0)+\Omega(0)B''(0))+
  \BigO(\alpha^q) 
  = \frac12\Omega''(0) + \BigO(\alpha^q),
$$
and
$$
  \varepsilon a_2(\alpha) = \frac1{3!}(\Omega'''B+3\Omega''B'+3\Omega'B''+
  \Omega B''')+
  \BigO(\alpha^q) 
  = \BigO(\alpha^q).
$$
For $r=3$ have already noted that initial data must be
chosen more carefully to have consistency. The value of $a_3(\alpha)$ 
must therefore be different from what is given in \eq{Fhom}.
Indeed,
\begin{align*}
  \varepsilon^2 a_3(\alpha) &= -\frac{\varepsilon^2}{4!}(\Omega''''B+4\Omega'''B'+6\Omega''B''+
  4\Omega' B'''+\Omega B'''')+
  \BigO(\alpha^q) \\
  &= 
  -\frac{\varepsilon^2}{4!}(\Omega''''(0)+6\Omega''(0)B''(0))+
  \BigO(\alpha^q).
\end{align*}
Hence, in total, using \theo{FW} and the expressions for $a_r$,
when $Q$ is a third order polynomial,
$$
    F_{\rm HMM}(\alpha) =
 \frac12\Omega''(0)Q'(0)
  - 
\frac{\varepsilon^2}{4!}\left[\Omega''''(0)+6\Omega''(0) B''(0)\right]{Q}'''(0)+\BigO(\alpha^{q}).
$$
The difference with \eq{Fhom} is in the 
second term, which comes in because the initial data is not
consistent on this level with the macro data.
The difference from the case of linear initial data, which
was analyzed in \cite{engquist2011}, is clearly seen; when $Q'''=0$ the
flux is $\BigO(\alpha^q)$ away from the homogenized flux as in \eqref{eq:FHMM_Q_alpha}.
However, we will see below in \sect{Consistency2} 
that if the initial data is made
consistent then the result agrees precisely with \eq{Fhom}
if $a_3(\alpha)$ is given by \eq{ahyp}.

Unfortunately, we are, not able to prove the full result 
in \eq{ahyp} at this point and
it stands as a conjecture. Nevertheless, we have some partial
results summarized in the theorem below. In particular we can show
that \eq{ahyp} is true for $r=0,1$. For $r\geq 2$
it is true for the first term in the infinite
series that defines $W(k,\alpha)$ in \eq{Wdef},
$$
   W_0(k,\alpha) := b_0(k)s_0(k,\alpha)w_0(k,\alpha).
$$
This term corresponds to the leading order part of the solution $v_0$ in \eq{vexact}.
We have
\begin{theorem}\lbtheo{Wtheo}
The function $W(\cdot,\alpha)$ belongs to 
$L^{2}(-\pi,\pi)\cap L^{\infty}(-\pi,\pi)$
with norms bounded independently of $\alpha\leq \alpha_0$.
When $K\in \mathbb{K}^{r,q}$ and $W(\cdot,\alpha)\in C^r(-\pi,\pi)$
we have for $r=0,1$,
$$
  W(0,\alpha) =0, \qquad
    W^{(1)}(0,\alpha) = 
    \frac{i}{2}\left.\frac{\D^{2}}{\D{k}^{2}}\Omega(k)B(k)\right|_{k=0}
   +\BigO(\alpha^q).
$$
and for $r\geq 2$,
\be{Wexpr}
   W_0^{(r)}(0,\alpha) = \frac{i}{r+1}\left.\frac{\D^{r+1}}{\D{k}^{r+1}}\Omega(k)B(k)\right|_{k=0}
   +\BigO(\varepsilon^{-r+1}\alpha^q).
\ee
\end{theorem}
To prove this we first need to establish some lemmas about
the coefficients $b_m$, $s_m$ and $w_m$, 
starting with the $b_m$ coefficients.
\begin{lem}\lblem{blem}
We have
$$
   \sum_{m=0}^\infty |b_m(k)|^2 = 1, \qquad
   b_m(0) = \delta_m, \qquad
     \sum_{m=0}^\infty b'_m(0)\psi_m(x,0)=-\partial_k\psi_0(x,0).
$$
Moreover, for the squared quantity $B(k)=|b_0(k)|^2$,
$$
   B(0)=1, \qquad B'(0)=0.
$$
\end{lem}
\begin{proof}
By definition,
$$
   1 = \sum_{m=0}^\infty b_m(k)\psi_m(x,k),
$$
since $\{\psi_m\}$ is an orthonormal $L^2$-basis for each fixed $k$.
Taking the $L^2$ norm over $[-1/2,1/2]$ we obtain the first result.
The second result is true since $\psi_0(x,0)\equiv 1$, and then
$$
   b_m(0) = \int_{-1/2}^{1/2}\psi_m^*(x,0)\D{x} = \int_{-1/2}^{1/2}\psi_m^*(x,0)
   \psi_0(x,0)\D{x} = \delta_m.
$$
Together with,
$$
  0 = \partial_k \sum_{m=0}^\infty b_m(k)\psi_m(x,k) = 
  \sum_{m=0}^\infty b_m'(k)\psi_m(x,k)
  + \sum_{m=0}^\infty b_m(k)\partial_k\psi_m(x,k),
$$
the third result follows. That $B(0)=1$ is true since $b_m(0)=\delta_m$.
For the last statement we note that for all $k$,
$$
   0=\frac{d}{dk}1=\frac{d}{dk}\int\psi_0(x,k)\psi^*_0(x,k)dx
   = 2\Re\int\psi_0(x,k)\partial_k\psi^*_0(x,k)dx.
$$
In particular, for $k=0$ we get $0=2\Re b'_0(0)$ since $\psi_0(x,0)\equiv 1$.
We finally note that $B'(0)=2\Re b_0(0)b'_0(0)$.
\end{proof}

\begin{lem}\lblem{slem}
Suppose $K\in \mathbb{K}^{p,q}$. 
For $s_m(k,\alpha)$ we have when $m=0$
$$
   \left.\frac{d^r}{dk^r}s_0(k,\alpha)\right|_{k=0} =\begin{cases}
   1, & r=0, \\
0, & 1\leq r\leq p,
   \end{cases}
$$
and for $m\geq 1$,
$$
   \left|s_m(0,\alpha)\right|\leq C\alpha^q, \qquad m\geq 1,
$$
where the constant $C$ is independent of $\alpha$ and $m$.
\end{lem}
\begin{proof}
Since $\omega_0(0)=0$,
$$
  s_0(0,\alpha) = \int K_\eta(t)\D{t} = 1.
$$
For $r\geq 1$, by the moment conditions \eq{momentcond},
\begin{align*}
   \left.\frac{d^r}{dk^r}s_0(k,\alpha)\right|_{k=0}
 &= \Re\int K_\eta(t) \left[c_{1,r}\frac{it}{\varepsilon}+c_{2,r}\frac{(it)^2}{\varepsilon^2}+\cdots +
 c_{r,r}\frac{(it)^{r}}{\varepsilon^{r}}\right]
   \exp(i\omega_0(0) t/\varepsilon)
   \D{t}\\
    &= \Re\int K_\eta(t) \left[c_{1,r}\frac{it}{\varepsilon}+c_{2,r}\frac{(it)^2}{\varepsilon^2}+\cdots +
 c_{r,r}\frac{(it)^{r}}{\varepsilon^{r}}\right]
   \D{t} = 0.
\end{align*}
Moreover, for $m\geq 1$, by \theo{oscapprox},
$$
 |s_m(0,\alpha)| = \left|\int K_\eta(t)
   \cos(\omega_m(0) t/\varepsilon)
   \D{t}\right| \leq C \left(\frac{2\pi \varepsilon}{\omega_m(0)\eta}\right)^q
   \leq C' \left(\frac{\alpha}{\omega_1(0)}\right)^q.
$$
\end{proof}


\begin{lem}\lblem{vlem}
Suppose $a(x)\in C^1$ and $K\in \mathbb{K}^{p,q}$ with $q\geq 1$. Then
$$
   \sup_{|k|\leq \pi} |w_m(k,\alpha)|\leq C,
$$
with $C$ independent of $m$ and $k$.
Let
$$
  h(x,\alpha) = -\alpha^2K'(\alpha x)a(x) 
-\alpha K(\alpha x)a'(x).
$$
Then $h\in C_c$ with $\supp\; h\subset[-1/\alpha,1/\alpha]$
and
we can write
$$
  w_m(k,\alpha) = 
    \int h(x,\alpha)\psi_m(x,k)e^{ikx} \D{x}.
$$
Moreover,
$$
   \sum_{m=0}^\infty \int_{-\pi}^{\pi}|w_m(k,\alpha)|^2 \D{k} 
   = \int |h(x)|^2\D{x}
\leq C(1+\alpha)^2,
$$
where the constant $C$ is independent of $k$ and $\alpha$.
In addition, $w_0(0,\alpha)=0$ and for $1\leq r\leq p$,
$$
w_0^{(r)}(0,\alpha) =
 \lim_{k\to 0}\frac{\D^r}{\D{k}^r}
 \frac{-\Omega(k)}{ik}b_0^*(k)+O(\varepsilon^{-r+1}\alpha^q).
$$
\end{lem}
\begin{proof}
Let 
$$
  \tilde{h}(x) :=
  -\varepsilon \partial_x K_\eta(x)a(x/\varepsilon)
=-\frac{\varepsilon}{\eta^2}K'(x/\eta)a(x/\varepsilon) 
-\frac{1}{\eta} K(x/\eta)a'(x/\varepsilon).
$$
Moreover, via intergration by parts,
\begin{align}\lbeq{vmexpr}
w_m(k,\alpha)&=
   \varepsilon \int K_\eta(x)a(x/\varepsilon)\partial_x \psi_m(x/\varepsilon,k)e^{ikx/\varepsilon} \D{x}\\
   &= \int \tilde{h}(x)
  \psi_m(x/\varepsilon,k)e^{ikx/\varepsilon} \D{x}.\nonumber
\end{align}
We have $\supp\;\tilde{h}\subset[-\eta,\eta]$ and $|\tilde{h}|\leq C(\alpha+1)/\eta$.
Therefore by Cauchy--Schwarz
\begin{align*}
  |w_m(k,\alpha)|^2 &\leq 
  \int_\eta^\eta |\tilde{h}(x)|^2 \D{x}\times
  \int_\eta^\eta |\psi_m(x/\varepsilon,k)|^2 \D{x}\\
  &\leq
  C \int_\eta^\eta |\alpha+1|^2/\eta^2 \D{x}\times
  \varepsilon\int_{1/\alpha}^{1/\alpha} |\psi_m(x,k)|^2\D{x}\\
  &\leq
  C \frac{|\alpha+1|^2}{\eta}
  \frac{\varepsilon}{\alpha} \leq C.
\end{align*}
This shows the boundedness.
Next, from \eq{vmexpr} it is clear that
\begin{align*}
w_m(k,\alpha) =
 \varepsilon \int \tilde{h}(\varepsilon x)
  \psi_m(x,k)e^{ikx} \D{x} =
  \int h(x,\alpha)
  \psi_m(x,k)e^{ikx} \D{x}.
\end{align*}
Hence, $w_m(k,\alpha)$ are the Bloch coefficients for $h(x,\alpha)$.
Therefore, by Parseval,
$$
   \sum_{m=0}^\infty \int_{-\pi}^{\pi}|w_m(k,\alpha)|^2 \D{k}
   = \int |h(x,\alpha)|^2\D{x}
   = \frac{1}{\alpha} \int_{-1}^1 |h(x/\alpha,\alpha)|^2\D{x}   
   \leq C(1+\alpha)^2.
$$
To prove the last result we let $\psi=\psi_0$ to simplify the notation.
Denoting the binomial cofficients by $c_{r\ell}$, we have
\begin{align*}
 \frac{\D^r}{\D{k}^r} w_0(k,\alpha) ={}&
 \frac{\D^r}{\D{k}^r} \varepsilon \int K_\eta(x)a(x/\varepsilon)\partial_x \psi(x/\varepsilon,k)e^{ikx/\varepsilon} \D{x}\\
 ={}&\varepsilon \sum_{\ell=0}^rc_{r\ell} \int K_\eta(x)a(x/\varepsilon)\partial_x \psi^{(\ell)}(x/\varepsilon,k)e^{ikx/\varepsilon} \left(\frac{ix}{\varepsilon}\right)^{r-\ell}\D{x}\\
 ={}&\sum_{\ell=0}^rc_{r\ell} \int K_\eta(x)a(x/\varepsilon) \psi_{x}^{(\ell)}(x/\varepsilon,k)e^{ikx/\varepsilon} \left(\frac{ix}{\varepsilon}\right)^{r-\ell}\D{x}
 \\
 &+ \sum_{\ell=0}^rc_{r\ell} \int K_\eta(x)a(x/\varepsilon) ik\psi^{(\ell)}(x/\varepsilon,k)e^{ikx/\varepsilon} \left(\frac{ix}{\varepsilon}\right)^{r-\ell}\D{x}\\
 &+\sum_{\ell=0}^{r-1}c_{r\ell} \int K_\eta(x)a(x/\varepsilon)\psi^{(\ell)}(x/\varepsilon,k)e^{ikx/\varepsilon} i(r-\ell)\left(\frac{ix}{\varepsilon}\right)^{r-\ell-1}\D{x}.
\end{align*}
Since $\psi(x,0)\equiv 1$,
\begin{align*}
  w_0^{(r)}(0,\alpha) ={}& 
\sum_{\ell=1}^rc_{r\ell} \int K_\eta(x)a(x/\varepsilon) 
\psi_{x}^{(\ell)}(x/\varepsilon,0)\left(\frac{ix}{\varepsilon}\right)^{r-\ell}\D{x}\\
&+
\sum_{\ell=0}^rc_{r\ell} \int K_\eta(x)a(x/\varepsilon) \psi^{(\ell)}(x/\varepsilon,0)i(r-\ell)\left(\frac{ix}{\varepsilon}\right)^{r-\ell-1}\D{x}.
\end{align*}
We can now use \eq{second} in \theo{oscapprox} to get
\begin{align*}
  w_0^{(r)}(0,\alpha) &= 
 \int_{-1/2}^{1/2}a(x) (\psi_x^{(r)}(x,0) +ic_{r,r-1}\psi^{(r-1)}(x,0))\D{x} + \BigO(\varepsilon^{-r+1}\alpha^q)\\
 &=
 \lim_{k\to 0} \frac{\D^r}{\D{k}^r}\int_{-1/2}^{1/2}a(x) (\psi_{x}(x,k) +ik\psi(x,k))\D{x} + \BigO(\varepsilon^{-r+1}\alpha^q)
\\
 &=
 \lim_{k\to 0} \frac{\D^r}{\D{k}^r}\frac{1}{ik}\int_{-1/2}^{1/2}(\partial_x + ik)a(x)(\partial_x+ik)\psi(x,k)\D{x} + \BigO(\varepsilon^{-r+1}\alpha^q)
 \\
 &=
 \lim_{k\to 0} \frac{\D^r}{\D{k}^r}\frac{-\omega_0(k)^2}{ik}\int_{-1/2}^{1/2}\psi(x,k)\D{x} + \BigO(\varepsilon^{-r+1}\alpha^q)
\\
 &=
 \lim_{k\to 0} \frac{\D^r}{\D{k}^r}\frac{-\omega_0(k)^2}{ik}b_0^*(k) + \BigO(\varepsilon^{-r+1}\alpha^q).
\end{align*}
%
This concludes the proof.
\end{proof}
\subsubsection{Proof of \theo{Wtheo}}

To show that $W\in L^2(-\pi,\pi)$ we note that by \lemlab{blem},
\lemlab{slem} and \lemlab{vlem}
\begin{align*}
\int_{-\pi}^{\pi} |W(k,\alpha)|^2 \D{k} &=
\int_{-\pi}^{\pi} \left|\sum_{m=0}^\infty b_m(k)s_m(k,\alpha)w_m(k,\alpha)\right|^2 \D{k} \\
&\leq
\int_{-\pi}^{\pi} \sup_m |s_m(k,\alpha)|^2 \sum_{m=0}^\infty |b_m(k)|^2
\sum_{m=0}^\infty |w_m(k,\alpha)|^2 \D{k} \\
&\leq
C\sum_{m=0}^\infty \int_{-\pi}^{\pi} |w_m(k,\alpha)|^2 \D{k} \leq C.\\
\end{align*}
We should also check that $W\in L^\infty$. Let
$$
g(x,k,\alpha)=      \sum_{m=0}^\infty b_m(k)s_m(k,\alpha)\psi_m(x,k).
$$
Then with $N=\lceil 1/\alpha\rceil$
$$
  W(k,\alpha) = \int_{-N}^N h(x,\alpha)g(x,k,\alpha) e^{ikx} \D{x}
$$
and since $|h|\leq C\alpha$ and by the periodicity of $g$,
$$
|W(k,\alpha)|\leq C\alpha \int_{-N}^{N} |g(x,k,\alpha)|\D{x}
\leq C\int_{-1/2}^{1/2} |g(x,k,\alpha)|\D{x}
\leq C||g(\cdot,k,\alpha)||_{L^2}.
$$
We note that since $\{\psi_m\}$ is an orthonormal $L^2$-basis for each fixed $k$
and $|s_m(k,\alpha)|\leq C$,
$$
 \int_{-1/2}^{1/2} |g|^2\D{x}  = \sum_{m=0}^\infty  
  |b_m(k)|^2|s_m(k,\alpha)|^2
  \leq  C^2\sum_{m=0}^\infty  
  |b_m(k)|^2  = C^2.
$$
Hence, $|W|\leq C$ uniformly in $k$.
This shows that $W\in L^\infty$ with a norm bounded uniformly in $k$ and $\alpha$.

We have by the lemmas above, and since $\omega_0(0,\alpha) \equiv 0$,
$$
  W(0,\alpha) = \sum_{m=0}^\infty b_m(0)s_m(0,\alpha)w_m(0,\alpha)
   = b_0(0)s_0(0,\alpha)w_0(0,\alpha)=0.
$$
For $1\leq r\leq p$, with binomial coefficients $c_{r\ell}$ we get
\begin{align*}
W_0^{(r)}(0,\alpha) &=
\left. \frac{\D^r}{\D{k}^r} b_0(k)s_0(k,\alpha)w_0(k,\alpha)\right|_{k=0} \\
 &= \sum_{\ell=0}^rc_{r\ell}s_0^{(r-\ell)}(0)
 \left. \frac{d^\ell}{\D{k}^\ell}b_0(k)w_0(k,\alpha)\right|_{k=0}\\
 &= 
 \left. \frac{\D^r}{\D{k}^r}b_0(k)w_0(k,\alpha)\right|_{k=0}
 \\
 &= \sum_{\ell=0}^rc_{r\ell}
 b_0^{(r-\ell)}(0)w_0^{(\ell)}(0,\alpha)
 \\
 &= \sum_{\ell=0}^rc_{r\ell}
\left. b_0^{(r-\ell)}(0)\left[\left(\frac{-\Omega(k)}{ik}b^*_0(k)\right)^{(\ell)}
\right|_{k=0}+O(\varepsilon^{-\ell+1}\alpha^q)\right]\\
 &= 
i\left. \frac{\D^r}{\D{k}^r}\frac{\Omega(k)}{k}|b_0(k)|^2
\right|_{k=0} + \BigO(\varepsilon^{-r+1}\alpha^q).
\end{align*}
Since if $g(k)=kf(k)$ we have $g^{(r)}(k)=kf^{(r)}(k)+rf^{(r-1)}(k)$,
$$
\left. \frac{\D^r}{\D{k}^r}\frac{\Omega(k)B(k)}{k}
\right|_{k=0}   =
\frac{1}{r+1}\left. \frac{\D^{r+1}}{\D{k}^{r+1}}\Omega(k)B(k)
\right|_{k=0},   
$$
which proves \eq{Wexpr}.
We should finally show that the difference $W^{(1)}-W_0^{(1)}$ is
small.
We set
$$
   z(x,k,\alpha) = \left[g(x,k,\alpha)-b_0(\alpha)s_0(k,\alpha)\psi_0(x,k)\right]e^{ikx}.
$$
Then
$$
  W(k,\alpha)-W_0(k,\alpha) = \int_{-N}^N h(x,\alpha)z(x,k,\alpha)\D{x}
$$
and since $b_m(0)=\delta_m$ by \lemlab{blem},
$$
  z_k(x,0,\alpha) = \sum_{m=1}^\infty b'_m(0)s_m(0,\alpha)\psi_m(x,0).
$$
As before,
\begin{align*}
  |W^{(1)}(k,\alpha)-W^{(1)}_0(k,\alpha)|&\leq 
  C\alpha\int_{-N}^N|z_k(x,k,\alpha)|\D{x}\leq
  C\int_{-1/2}^{1/2}|z_k(x,k,\alpha)|\D{x}\\
  &\leq
  C||z_k(\cdot,0,\alpha)||_{L^2}.
\end{align*}
Furthermore,
\begin{align*}
||z_k(\cdot,0,\alpha)||_{L^2}^2
&=\sum_{m=1}^\infty |b'_m(0)|^2|s_m(k,0)|^2
\leq C\alpha^{2q}\sum_{m=1}^\infty |b'_m(0)|^2\\
&\leq C\alpha^{2q}||\partial_k\psi_0(\cdot,0)||_{L^2}^2
\leq C\alpha^{2q}.
\end{align*}
Hence, the size of the remaining sum is of the same order as the 
error $O(\varepsilon^{-r+1}\alpha^q)$ in the leading term for $r=1$.

\subsection{Consistency}\lbsec{Consistency2}

In this section we make a formal derivation to motivate the relationship \eq{MQ}.
We therefore
consider the exact solution for initial data $Q(x)$ which was given above
in \eq{vexact}.
We assume that only the zeroth term $v_0$ is relevant and 
that the initial data $Q$ is band limited as in \eq{compactsupp}.
Then, by \lemlab{Fflemma} and \eq{vexact},
$$
  \veps(t,x) \approx v_0(t,x) =
  \int_{-1/2\varepsilon}^{1/2\varepsilon}
  \hat{Q}(k)b_0(\varepsilon k)
    \psi_0(x/\varepsilon,\varepsilon k)
 \exp(ikx)\cos(\omega_0(\varepsilon k)t/\varepsilon)
dk.
$$
This implies that
$$
  ({\mathcal M}Q)(x) = ({\mathcal K} \veps)(0,x) \approx \int K_\eta(x'-x)\tilde{U}(x',x'/\varepsilon)dx',
$$
where
$$
  \tilde{U}(x,y) = 
  \int_{-L}^{L} 
  \hat{Q}(k)b_0(\varepsilon k)s_0(\varepsilon k, \alpha)
    \psi_0(y,\varepsilon k)
 \exp(ikx)dk.
$$
Since $k$ is in a bounded set we can Taylor expand the terms in $\varepsilon k$.
From \lemlab{slem},
$$
  s_0(\varepsilon k, \alpha) = 1 + O\left((\varepsilon k)^{p+1}\right).
$$
Hence, 
\begin{align*}
  \tilde{U}(x,y) 
  &= 
  \sum_{\ell=0}^p
  \left.\frac{d^\ell}{dk^\ell}
  b_0(k)
    \psi_0(y,k)\right|_{k=0}\int_{-L}^{L}
    \frac{(\varepsilon k)^{\ell}}{\ell !}
  \hat{Q}(k)
 \exp(ikx)dk
+ \BigO(\varepsilon^{p+1})
\\  
&= 
  \sum_{\ell=0}^p
  \frac{(-i\varepsilon)^{\ell}}{\ell !}
  \left.\frac{d^\ell}{dk^\ell}
  b_0(k)
    \psi_0(y,k)\right|_{k=0}
    Q^{(\ell)}(x)
+ \BigO(\varepsilon^{p+1}).
\end{align*}
We can use \theo{oscapprox} to deduce that
\begin{align*}
   \int K_\eta(x'-x)\partial_k^r\psi_0(x'/\varepsilon,0)Q^{(\ell)}(x')\D{x'} 
   &=
   \int_{-1/2}^{1/2}\partial_k^r\psi_0(x,0)dx Q^{(\ell)}(x) +
   \BigO(\alpha^q)\\
   &=
   \partial_k^rb_0(0)^*Q^{(\ell)}(x) +
   \BigO(\alpha^q).
\end{align*}
Therefore, 
$$
({\mathcal M}Q)(x) \approx 
  \sum_{\ell=0}^p
  \frac{(-i\varepsilon)^{\ell}}{\ell !}
  B^{(\ell)}(0)
    Q^{(\ell)}(x) +   \BigO(\varepsilon^\ell\alpha^q)
+ \BigO(\varepsilon^{p+1}).
$$
Since $B(0)=1$ and $B'(0)=0$ by \lemlab{blem}
we have, in particular when $p\geq 2$
$$
({\mathcal M}Q)(x) \approx Q(x) 
  -\frac{\varepsilon^{2}}{2}
  B''(0)
    Q''(x) +   \BigO(\varepsilon^3 + \alpha^q).
$$
This is the expression \eq{MQ} with $\gamma=-B''(0)/2$.
In order to have an accuracy that is better than an $O(\varepsilon^2)$ 
we need
to correct $Q(x)$ and to use a kernel such that $\alpha^q=O(\varepsilon^3)$.
More precisely, we should take
$$
  \tilde{Q}(x) = Q(x)+\frac{\varepsilon^{2}}{2} B''(0)Q''(x),
$$
and use this as initial data instead of $Q$. Then, if $Q$ is a third order polynomial,
\begin{align*}
({\mathcal K}u)(0,x) &\approx
\tilde{Q}(x) 
  -\frac{\varepsilon^{2}}{2}
  B''(0)
    \tilde{Q}''(x) +   \BigO(\varepsilon^3)
    \\
&=\left(Q(x)+\frac{\varepsilon^{2}}{2} B''(0)Q''(x)\right)
  -\frac{\varepsilon^{2}}{2}
  B''(0)
Q''(x)+   \BigO(\varepsilon^3)
\\
&=Q(x)+O(\varepsilon^3).
\end{align*}
Let us now see what the implications are for the computation of
the flux $F_{\rm HMM}$. If $Q(x)$ is a third order polynomial then so
is $\tilde{Q}(x)$ and we can use the result in the previous section.
With initial data $\tilde{Q}(x)$ we obtain
\begin{align*}
  F_{\rm HMM} &=\frac12
\Omega'' \tilde{Q}'(0)
  - 
\varepsilon^2\frac{\Omega''''+6\Omega'' B''}{4!}\tilde{Q}'''(0)+O(\alpha^{q})\\
&=
 \frac12\Omega''\left(Q'(0)+\frac{\varepsilon^{2}}{2} B''Q'''(0)\right)
  - 
\varepsilon^2\frac{\Omega''''+6\Omega'' B''}{4!}
Q'''(0)+O(\alpha^{q})\\
&=
 \frac12\Omega''Q'(0)
  - 
\varepsilon^2\frac{\Omega''''}{4!}Q'''(0)
+O(\varepsilon^3).
\end{align*}
This is thus consistent with \eq{Fhom}  upto order $O(\varepsilon^3)$
if $\alpha^q=O(\varepsilon^3)$.

\section{Numerical Examples \label{section:num}}

In this section we consider three long time wave propagation problems and a detailed convergence study of the flux correction.
The three problems are of the form,
\begin{equation}
	\left\{
	\begin{aligned}
		& u^{\varepsilon}_{tt} - \partial_x (A^{\varepsilon} u^{\varepsilon}_x) = 0, \qquad [0,1] \times [0,T^{\varepsilon}], \\
		& u^{\varepsilon}(x,0) = e^{-100 x^2} + e^{-100(1-x)^2}, \quad u^{\varepsilon}_t(x,0) = 0, \qquad \forall x \in [0,1],
	\end{aligned}
	\right.
	\label{eq:num:exact}
\end{equation}
with $1$-periodic boundary conditions in $x$ and
$T^{\varepsilon} = \mathcal{O}(\varepsilon^{-2})$. 
Note that the initial data has no $\varepsilon$ dependency. 
We will try three different $A^{\varepsilon}=A^{\varepsilon}_k$ parameters:
$$
	\left\{
	\begin{aligned}
		& A^{\varepsilon}_1(x) = 1.1 + \sin \frac{2 \pi x}{\varepsilon}, \\
		& A^{\varepsilon}_2(x) = 1.1 + \frac{1}{2} \left( \cos 2 \pi x + \sin \frac{2 \pi x}{\varepsilon} \right), \\
		& A^{\varepsilon}_3(x) = \left(1.1+\sin{}2\pi{}\frac{x}{\varepsilon}\right)\left( 1.5+0.5(\cos(2\pi{}x-1))\right). \\
	\end{aligned}
	\right.
$$
The problems can be expected to have dispersive effects which are not described by the homogenized wave equation \eqref{eq:introduction:wavebar} restated here for convenience,
\begin{equation}
	\left\{
	\begin{aligned}
		& \bar{u}_{tt} - \partial_x (\bar{A} \bar{u}_x) = 0, \qquad [0,1] \times [0,T], \\
		& \bar{u}(x,0) = e^{-100 x^2} + e^{-100(1-x)^2}, \quad \bar{u}_t(x,0) = 0, \qquad \forall x \in [0,1],
	\end{aligned}
	\right.
	\label{eq:num:hom}
\end{equation}
where $\bar{u}$ is $1$-periodic. 
We suppose there is an effective equation valid for $T^{\varepsilon}=\mathcal{O}(\varepsilon^{-2})$ of the form
\begin{equation}
	\left\{
	\begin{aligned}
		& \tilde{u}_{tt} - \partial_x \left( \bar{A}(x) \tilde{u}_x + \beta(x) \varepsilon^2 \tilde{u}_{xxx} \right)= 0, \qquad [0,1] \times [0,T^{\varepsilon}], \\
		& \tilde{u}(x,0) = e^{-100 x^2} + e^{-100 (1-x)^2}, \quad \tilde{u}_t(x,0) = 0, \qquad \forall x \in [0,1], \\
	\end{aligned}
	\right.
	\label{eq:num:eff}
\end{equation}
where $\bar{A}$ is the same $\bar{A}$ as in \eqref{eq:num:hom}, but unfortunately as we described in relation to \eqref{eq:introduction:eff}, $\beta$ can be difficult to determine from $A$, both symbolically and numerically. We compute $\bar{A}(x)$ and $\beta(x)$ from a $A^{\varepsilon}$ of the form $A^{\varepsilon}(x)=A(x,x/\varepsilon)$ by freezing $x$ and computing $\bar{A}(x)$ and $\beta(x)$ as if $A^{\varepsilon}$ had only fast oscillations. We have used Maple to compute the $\beta$ coefficients from $A$, \cite{holst2011}.

We will consider four methods for each wave propagation problem: An exact (DNS) solution of (\ref{eq:num:exact}) where we discretized the full problem with a finite difference method; a discretization of (\ref{eq:num:hom}), the corresponding homogenized equation (HOM); a discretization of 
(\ref{eq:num:eff}), the corresponding effective equation (EFF); and finally a HMM solution.
Throughout our examples we will use the following parameters: $\varepsilon = 0.03$, $\eta=\tau=20\varepsilon$ and a polynomial kernel which is 19 times continuously differentiable and has 19 zero moments.  We will also use the notation,
$$
	\lambda := \frac{\Delta t}{\Delta x}, \quad \rho^{\varepsilon} := \frac{\varepsilon}{\Delta x}, \quad \rho := \frac{1}{\Delta x}.
$$


\subsection{Wave Propagation Problem One} \label{section:num:ex1}


We consider $A^{\varepsilon}_1(x) = 1.1 + \sin 2 \pi x/\varepsilon$, which has only a fast scale. For the $A^{\varepsilon}_1$ at hand we have $\bar{A}_1=\sqrt{0.21}$. In general, for $A^{\varepsilon}(x)=A(x/\varepsilon)$ where $A(y)=\alpha+\beta\sin{}2\pi y$ we have that $\bar{A}=\sqrt{\alpha^2-\beta^2}$. 
A plot of $\bar{A}_1(x)$ is shown in Fig.~\ref{fig:material1} and $\beta$ is constant $\beta=0.01078280318$.
The parameters to this problem: $T_{\text{max}} = 12.4976$ (2000 HMM steps) and for all solvers except the exact finite difference solutions (direct numerical simulation, e.g., DNS) we used $\lambda = 0.5, \rho = 80$. For DNS we used $\lambda = 0.5$ and $\rho^{\varepsilon} = 64$.
See Fig.~\ref{fig:sol1} for a plot of the numerical solutions.


\begin{figure}
	\centering
	\includegraphics[width=.5\textwidth]{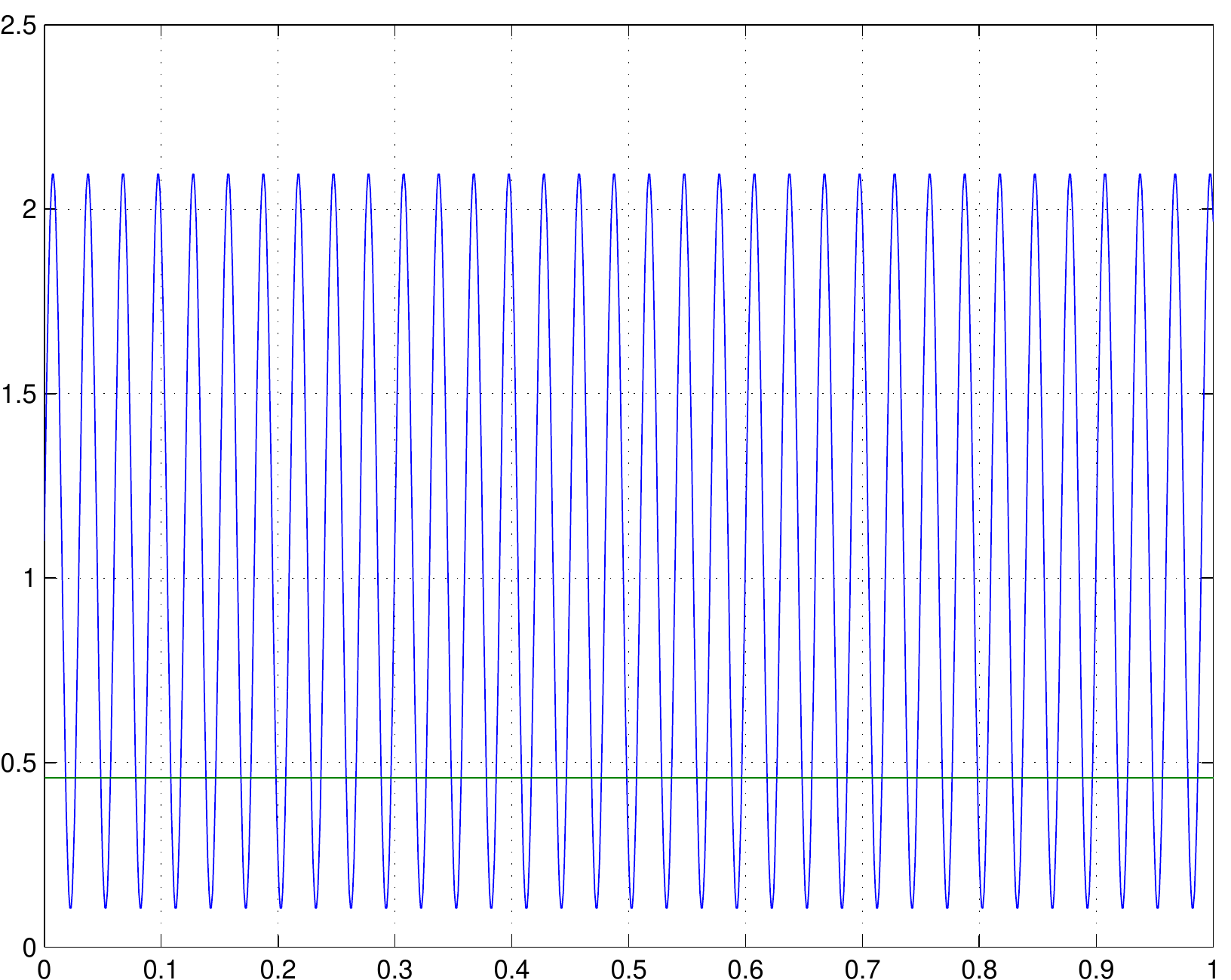}\\
	\captionof{figure}{Example One: Here we have $A^{\varepsilon}_1(x)$ plotted with the constant $\bar{A}_1(x)=\sqrt{0.21}$.}
	\label{fig:material1}
\end{figure}

\begin{figure}
	\centering
	\includegraphics[width=\textwidth]{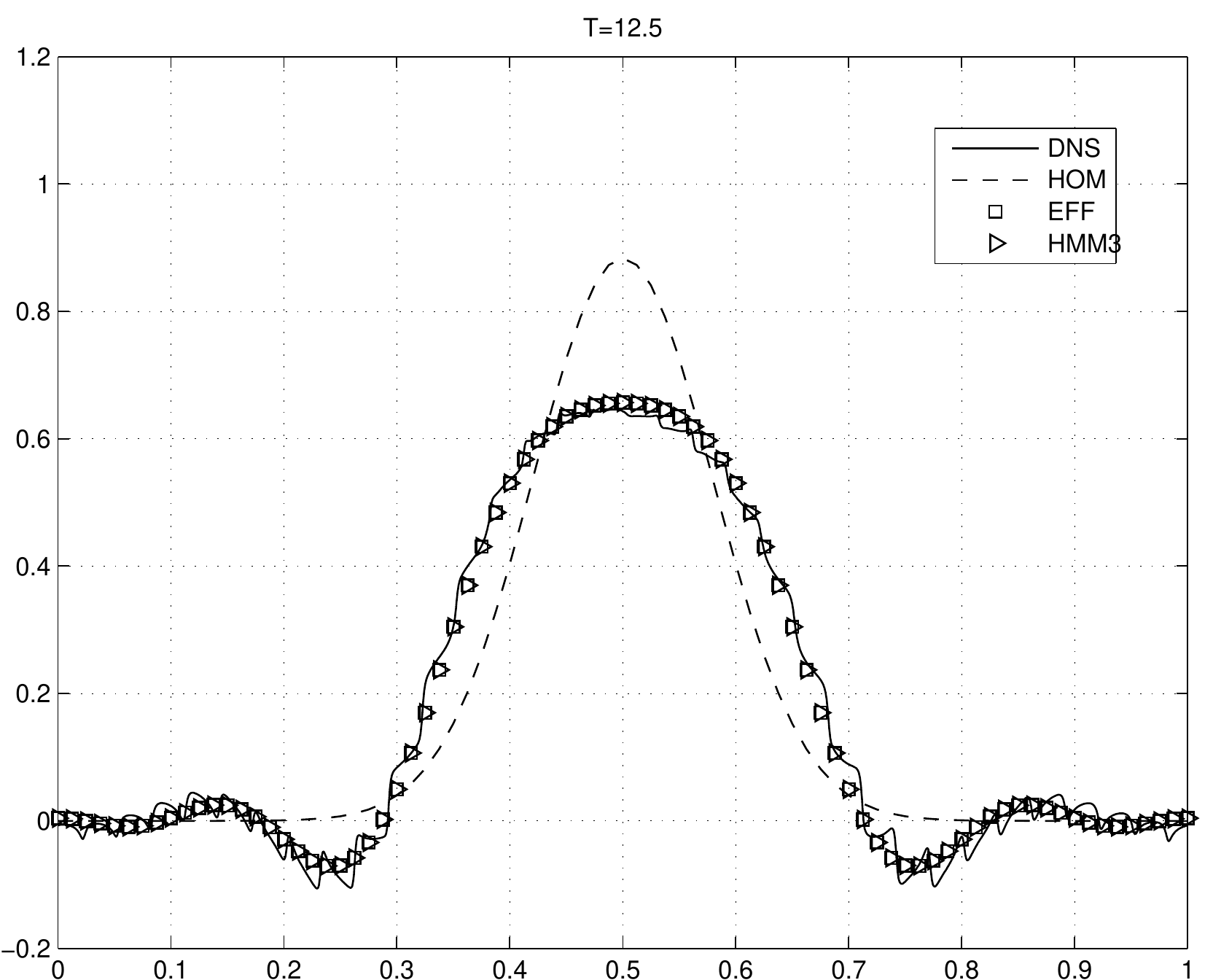}
	\caption{Example One: HMM solution of \eqref{eq:num:exact} together with its effective and homogenized equations \eqref{eq:num:hom} and \eqref{eq:num:eff}.}
	\label{fig:sol1}
\end{figure}

\subsection{Wave Propagation Problem Two} \label{section:num:ex2}

We consider $A^{\varepsilon}(x) = 1.1 + \frac{1}{2} \left( \cos 2 \pi x + \sin 2 \pi x/\varepsilon \right)$. The homogenized $\bar{A}$ can be computed analytically by freezing $x$: $\bar{A}(x) = \left(\int_0^1 \tfrac{1}{A(x,y)}\D{y}\right)^{-1}$. We have that $\bar{A}_2(x) = \sqrt{(0.5 \cos(2 \pi x) + 1.1)^2 - 0.25}$. The profile $\bar{A}_2(x)$ and $\beta_2(x)$ is shown in Fig.~\ref{fig:material2}. 

Numerical parameters: We have $T_{\text{max}} = 7.9985$ (1600 time steps for HMM), $\lambda = 0.5$ and $\rho = 100$ for all solvers except the DNS solver which uses $\lambda = 0.25$ and $\rho^{\varepsilon} = 64$.
See Fig.~\ref{fig:sol2} for a plot of the numerical solutions.
	
	
\begin{figure}
	\centering
	\includegraphics[width=.5\textwidth]{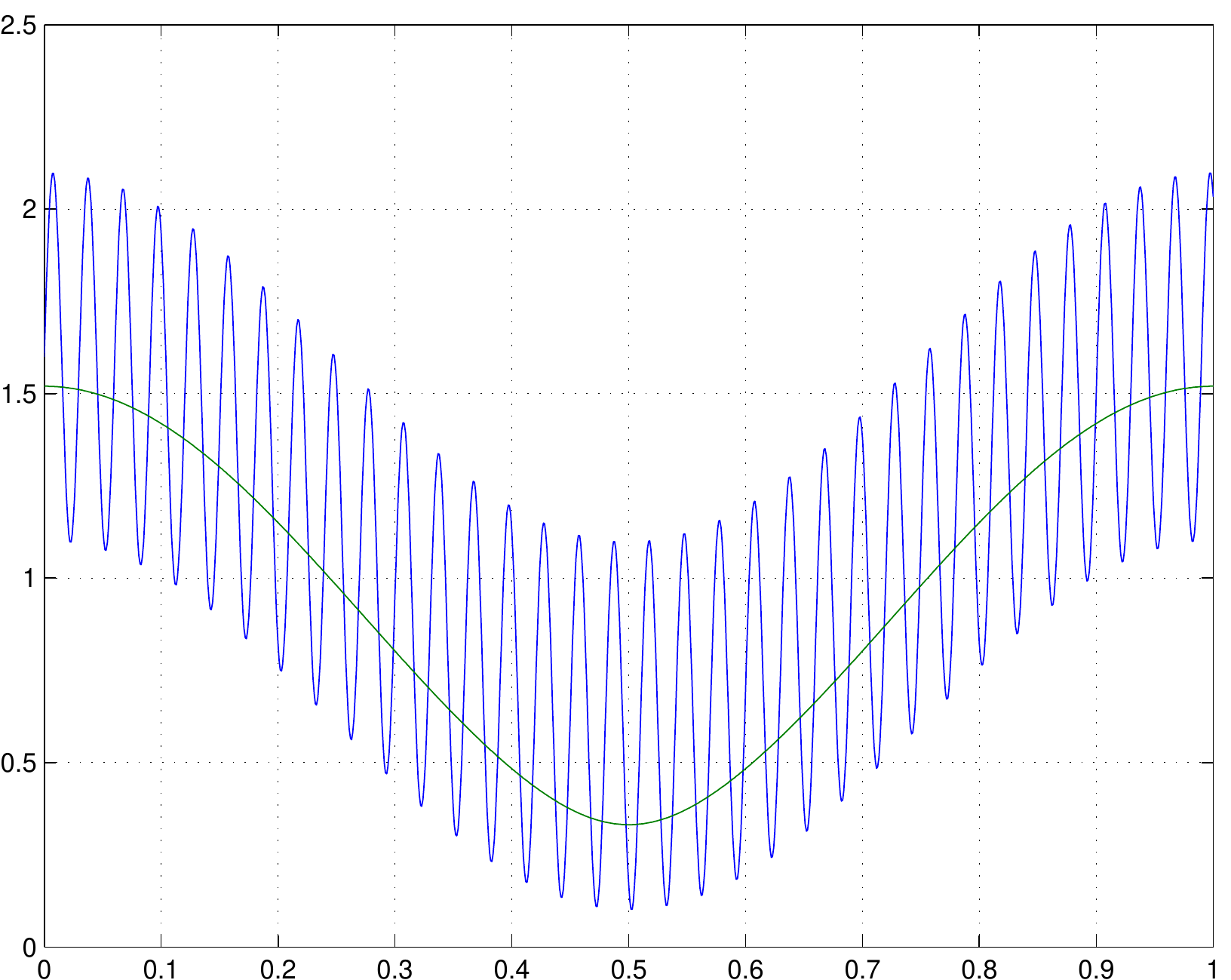}%
	\includegraphics[width=.5\textwidth]{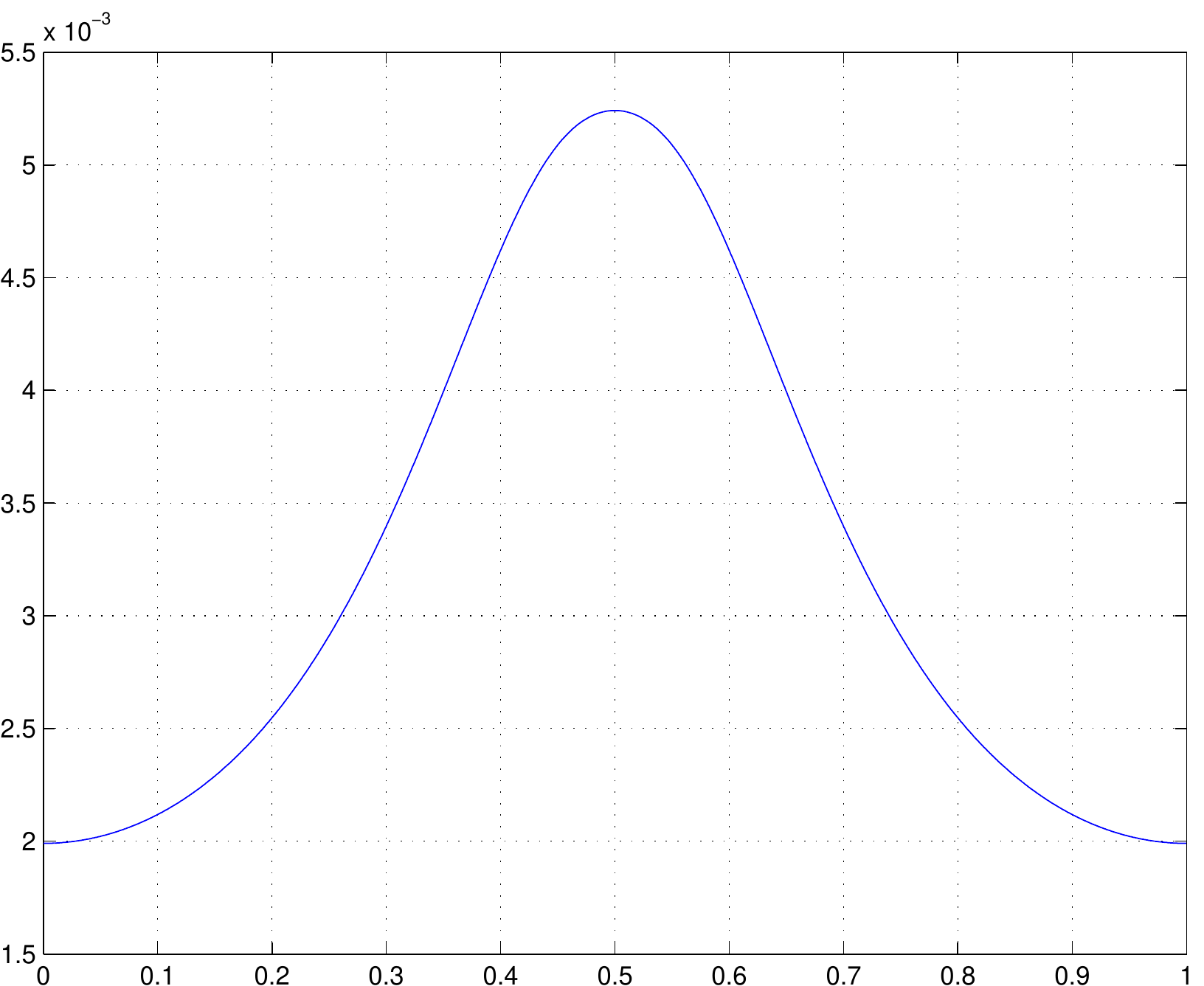}\\
	\captionof{figure}{Example Two: To the left we have $A^{\varepsilon}_2(x)$ plotted with $\bar{A}_2(x)$ and to the right we see $\beta_2(x)$.}
	\label{fig:material2}
\end{figure}

\begin{figure}
	\centering
	\includegraphics[width=\textwidth]{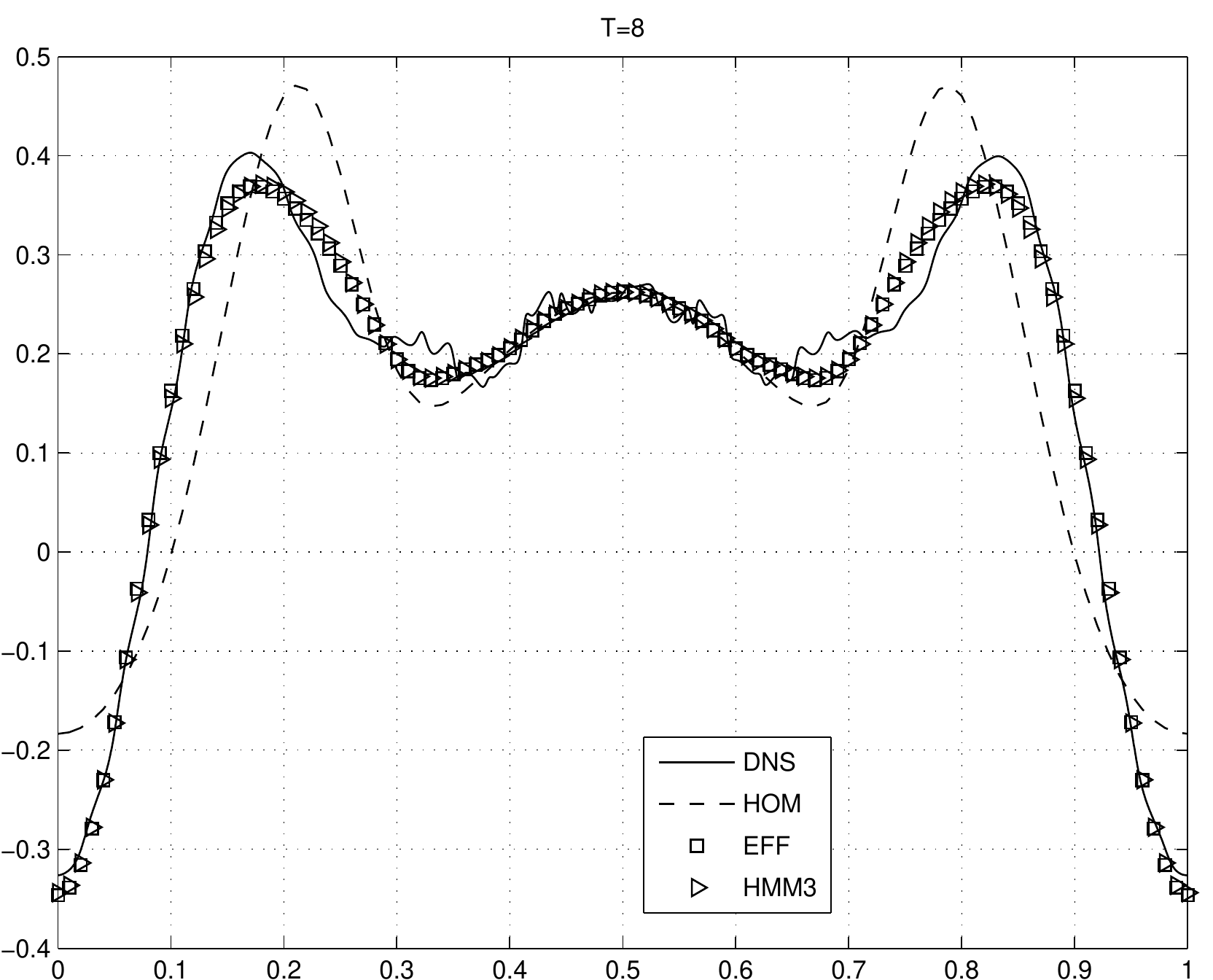}
	\caption{Example Two: HMM solution of \eqref{eq:num:exact} together with its effective and homogenized equations \eqref{eq:num:hom} and \eqref{eq:num:eff}.}
	\label{fig:sol2}
\end{figure}

\subsection{Wave Propagation Problem Three} \label{section:num:ex3} 

Finally we consider $A^{\varepsilon}_3(x)=\left(1.1+\sin{}2\pi{}\frac{x}{\varepsilon}\right)\left( 1.5+0.5(\cos(2\pi{}x-1))\right)$ which has the homogenized $\bar{A}_3(x)=\sqrt{0.21}(1.5+0.5(\cos(2\pi{}x)-1))$. The profile $\bar{A}_3(x)$ and $\beta_3(x)$ is shown in Fig.~\ref{fig:material4}. 

Numerical parameters: We have $\varepsilon = 0.03$ and $T_{\text{max}} = 8.5677$ (1200 HMM steps). For all sovers we used $\lambda = 0.5$ and $\rho = 70$ except for the DNS solver which used $\lambda = 0.25$ and $\rho^{\varepsilon} = 64$.
See Fig.~\ref{fig:sol4} for a plot of the numerical solutions.


\begin{figure}
	\centering
	\includegraphics[width=.5\textwidth]{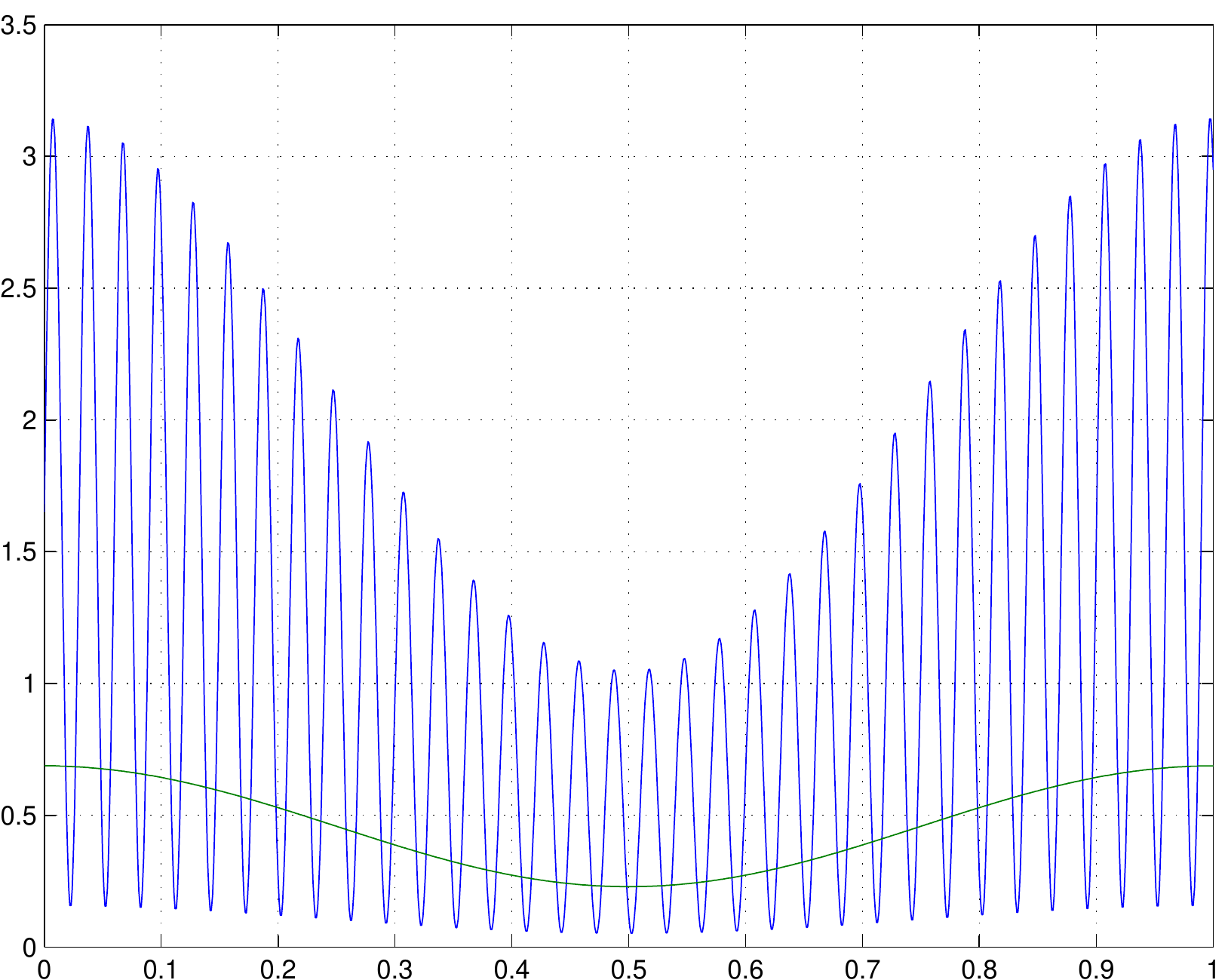}%
	\includegraphics[width=.5\textwidth]{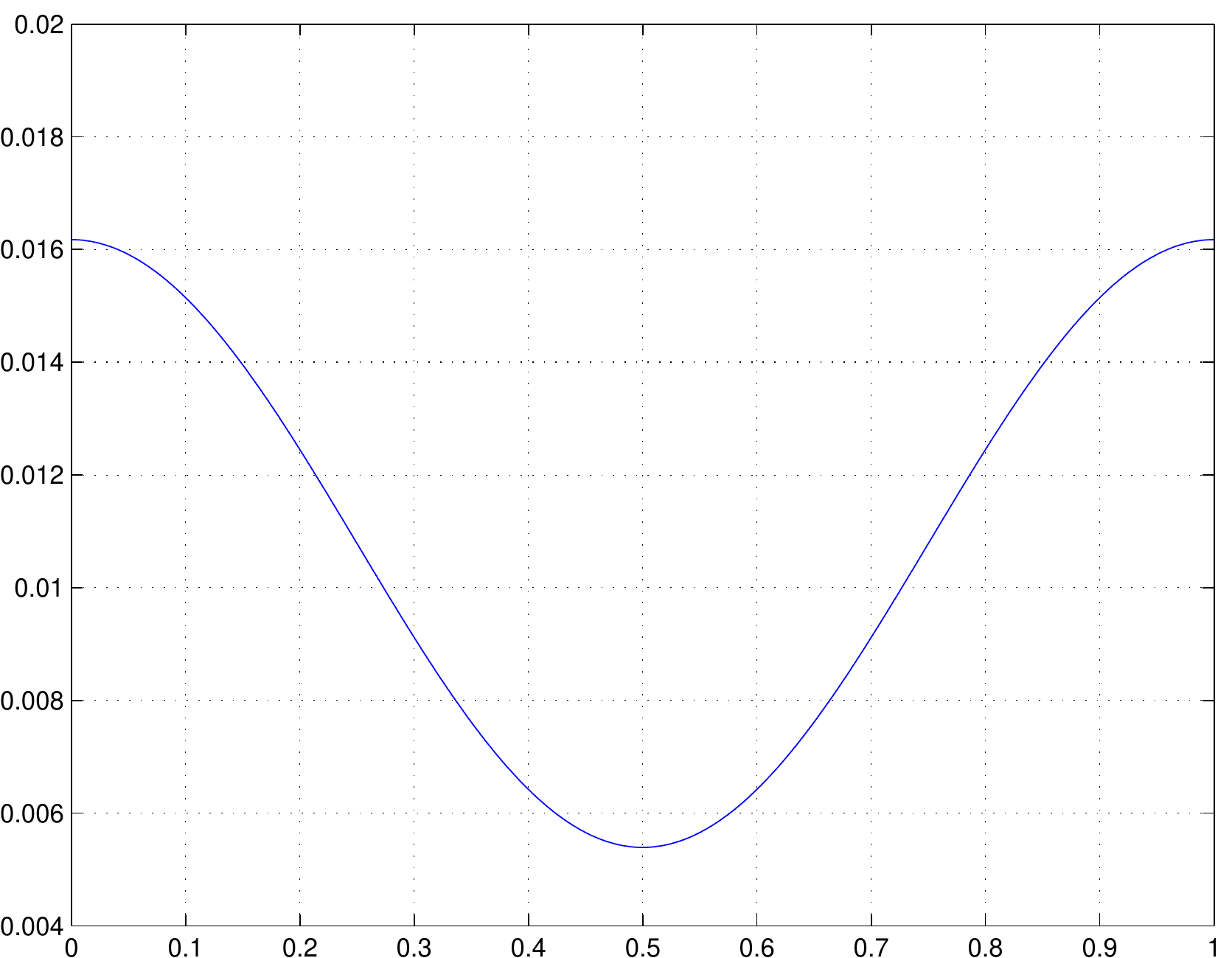}\\
	\captionof{figure}{Example Three: To the left we have $A^{\varepsilon}_3(x)$ plotted with $\bar{A}_3(x)$ and to the right we see $\beta_3(x)$.}
	\label{fig:material4}
\end{figure}

\begin{figure}
	\centering
	\includegraphics[width=\textwidth]{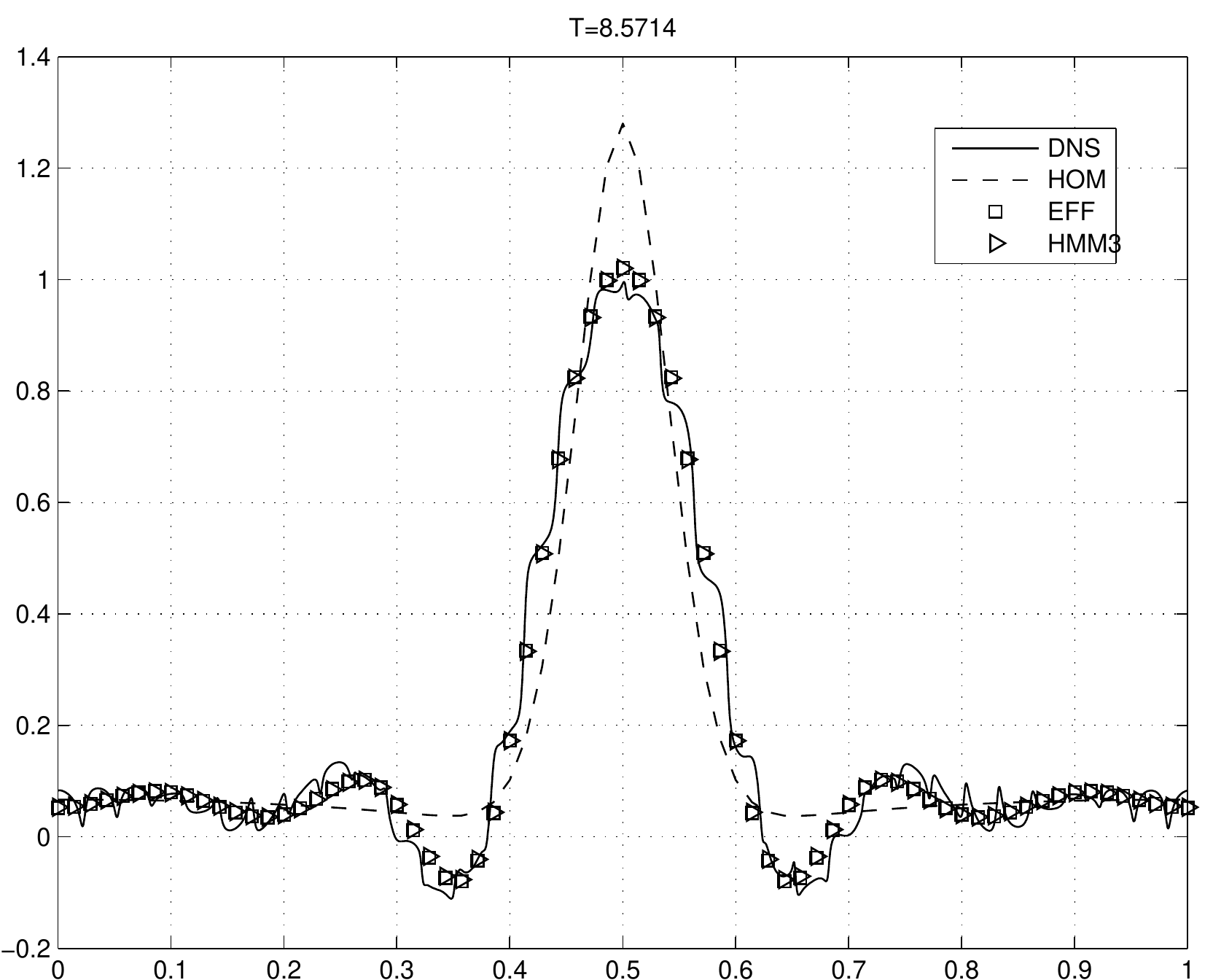}
	\caption{Example Three: HMM solution of \eqref{eq:num:exact} together with its effective and homogenized equations \eqref{eq:num:hom} and \eqref{eq:num:eff}.}
	\label{fig:sol4}
\end{figure}

In conclusion: We see that our numerical method accurately captures the solution of the effective equation. The choice of a bigger $\varepsilon$ is such that the model error in some cases is quite big. The difference between the exact and HMM solution and the solution of the effective equation will vanish as $\varepsilon \rightarrow 0$. These illustrations show that our HMM method gives comparable results to the solution of an effective equation when it is known. But as we argued before; the effective equation might not be known.

\subsection{Numerical Investigation of the HMM Flux} \lbsec{NumFlux}

We will now demonstrate why the correction is essential to obtain a useful approximation to the long time wave propagation problem. We assume the standard wave propagation problem \eqref{eq:introduction:wave} where $A^{\varepsilon}_1(x) = 1.1 + \sin 2 \pi \tfrac{x}{\varepsilon}$ and $\varepsilon=0.01$. We can then compute the correction matrix $M$ and the uncorrected flux $F$, using the same micro solver parameters $\tau=\eta=20\varepsilon$, $\rho^{\varepsilon}=64$, $\lambda=0.5$ and kernel from $\mathbb{K}^{9,9}$ and five sample points. We get $[\tfrac{1}{\varepsilon^2}(M-I)|\f]$
\begin{equation}
\left[
\text{\begin{scriptsize}
\begin{tabular}{rrrr|r}
  -2.4425e-11&  -2.3810e-09&   4.7060e-02&  -2.8439e-10& 2.303053047621222e-14\\
  -1.3116e-09&   1.5765e-10&  -3.3997e-08&   1.4118e-01& 4.582575696102362e-01\\
   5.2599e-07&   3.2966e-04&   7.4096e-09&   3.9335e-05& 9.160534264502929e-18\\
   1.4890e-05&  -1.9744e-05&   2.0390e-03&  -2.3307e-06& 1.293944940509046e-05\\
\end{tabular}
\end{scriptsize}}
\right].
\end{equation}

The corrected flux $\tilde{\f}$ is then (c.f. \eqref{eq:tildef}),
\begin{equation}
\tilde{\f} = M^{-T} \f = 
\left[
\text{\begin{scriptsize}
\begin{tabular}{r}
    -2.083614204143469e-14\\
     4.582575696102412e-01\\
     2.386730479027304e-13\\
     6.469733186662828e-06\\
\end{tabular}
\end{scriptsize}}
\right],
\end{equation}
and the exact flux $\f^{\text{exact}}$ is
\begin{equation}
	\begin{aligned}
		\f^{\text{exact}} &= 
		\left[
		\begin{tabular}{rrrr}
			0 \\ $\sqrt{0.21}$ \\ 0 \\ $6 \beta \varepsilon^2$ 
		\end{tabular}
		\right],
		\qquad \beta = 0.01078280318\ldots \\		
		& = \left[
		\text{\begin{scriptsize}
		\begin{tabular}{rrrr}
			0 \\ 4.58257569495584e-01 \\ 0 \\ 6.46968190800000e-06 
		\end{tabular}
		\end{scriptsize}}
		\right].
	\end{aligned}
\end{equation}
As we can see the correction affects mostly the last component of $\tilde{\f}$ related to the dispersive effects and not so much the lower order terms. 
The relative error $(f_1-f^{\text{exact}}_1)/f^{\text{exact}}_1$ is a mere 2.5019e-10 but the error $(f_3-f^{\text{exact}}_3)/f^{\text{exact}}_3$ is 1.0000, i.e. zero correct digits.

To conclude our investigation of the HMM fluxes we consider the second material $A^{\varepsilon}_2(x)=1.1 + \frac{1}{2} \left( \cos 2 \pi x + \sin \frac{2 \pi x}{\varepsilon} \right)$ which features both fast oscillations and slowly varying parts. We will visualize the correction in the HMM procedure,
$$
	\frac{|\tilde{f}_i-f_i|}{\varepsilon^2}, \qquad 1 \leq i \leq 3,
$$
for the components 1, 2 and 3 corresponding to $Q(x)=x^i$, $1\leq i\leq 3$ in the macroscopic points $x=0$ and $x=0.3$. 
We also check the difference between the HMM corrected
flux and the flux of the effective equation (\ref{eq:num:eff}) where $\beta(x)$
was computed by freezing the slow variation in $A^\varepsilon_2$, i.e.
$$
	\frac{|\tilde{f}_1-\bar{A}(x)|}{\varepsilon^2}, \qquad
	\frac{|\tilde{f}_2|}{\varepsilon^2}, \qquad
	\frac{|\tilde{f}_3-6\varepsilon^2\beta(x)|}{\varepsilon^2}.
$$
We let $\varepsilon$ vary between $0.1$ and $0.001$ in finite steps. We scale $\eta$ and $\tau$ such that $\eta=\tau=20\varepsilon$. In this experiment we used 16 points per $\varepsilon$ for the the micro solver, e.g., $\rho^{\varepsilon}=16$. See Figs.~\ref{fig:flux1}, \ref{fig:flux2} and \ref{fig:flux3} for plots of the convergence behavior. Note that $\tilde{f}_0$ and $f_0$ are both zero and therefore not plotted.
\begin{figure}[h!t]
\begin{center}
	\includegraphics[width=0.5\textwidth]{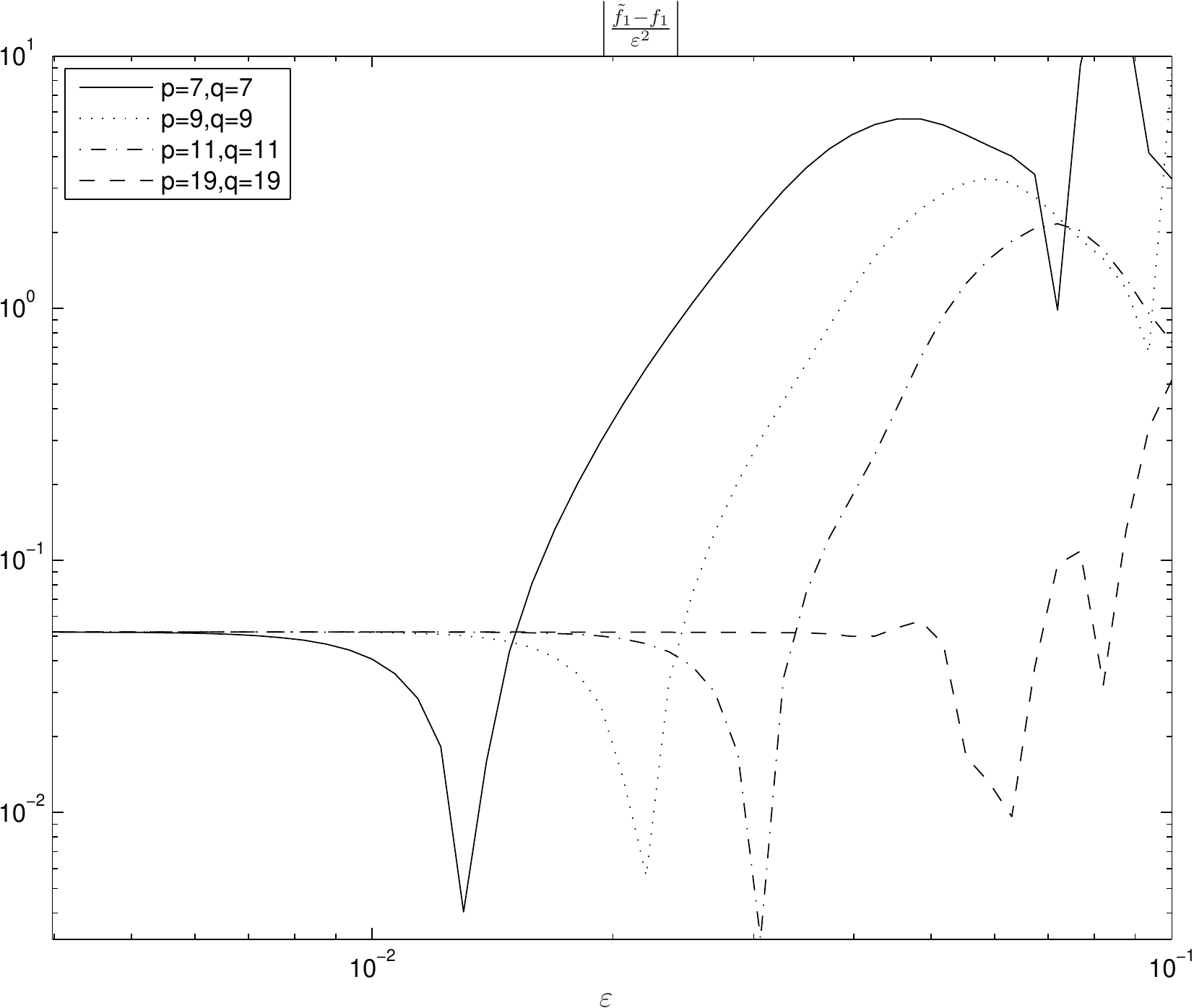}%
	\includegraphics[width=0.5\textwidth]{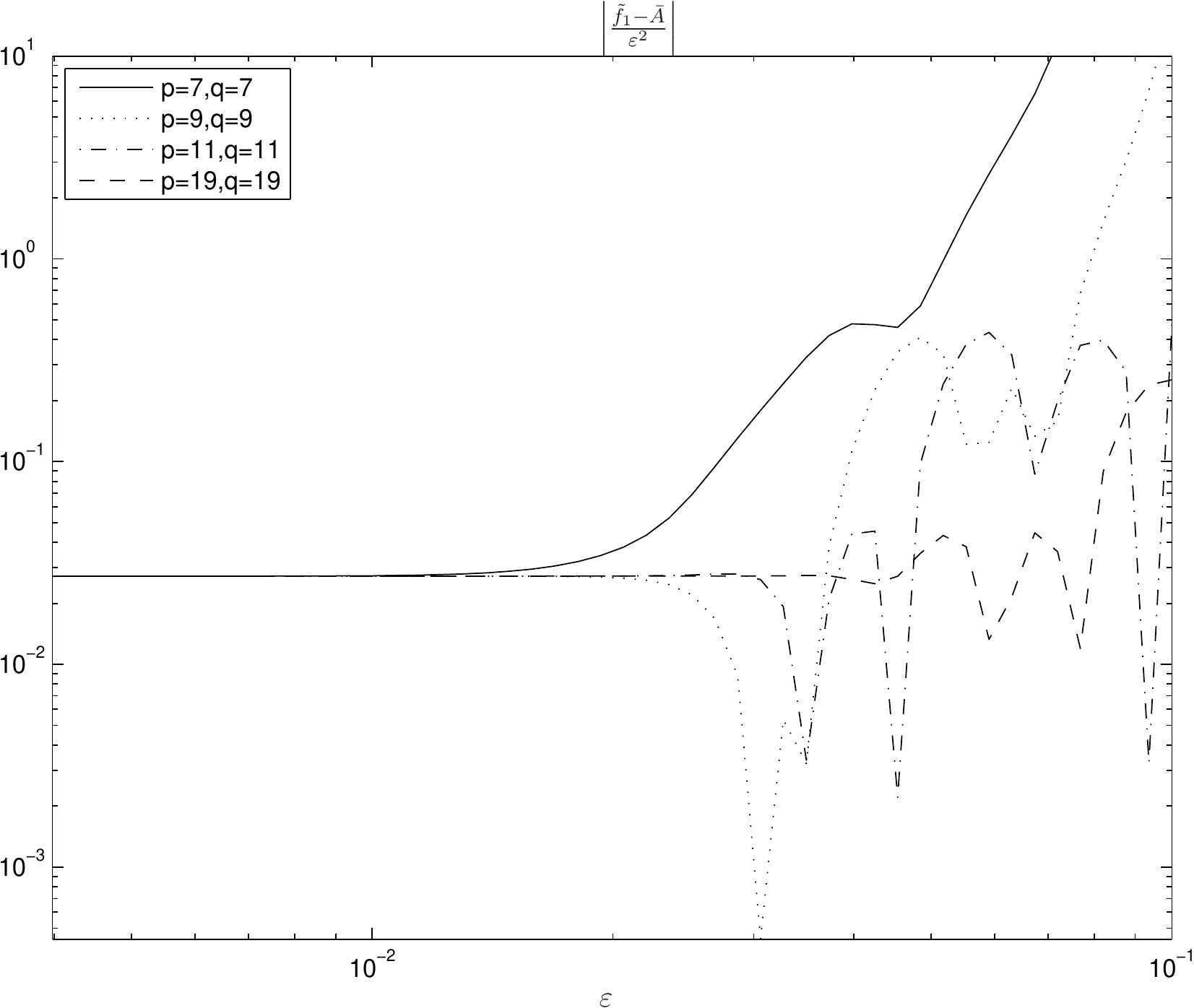}%
	\captionof{figure}{A log-log plot of the difference $\frac{|\tilde{f}_1-f_1|}{\varepsilon^2}$ (left) and $\frac{|\tilde{f}_1-\bar{A}|}{\varepsilon^2}$ (right) at $x=0$ for the material $A^{\varepsilon}_2$, 
	using kernels with various $p$- and $q$-parameters.}
	\label{fig:flux1}
\end{center}
\end{figure}
\begin{figure}[h!t]
\begin{center}
	\includegraphics[width=0.5\textwidth]{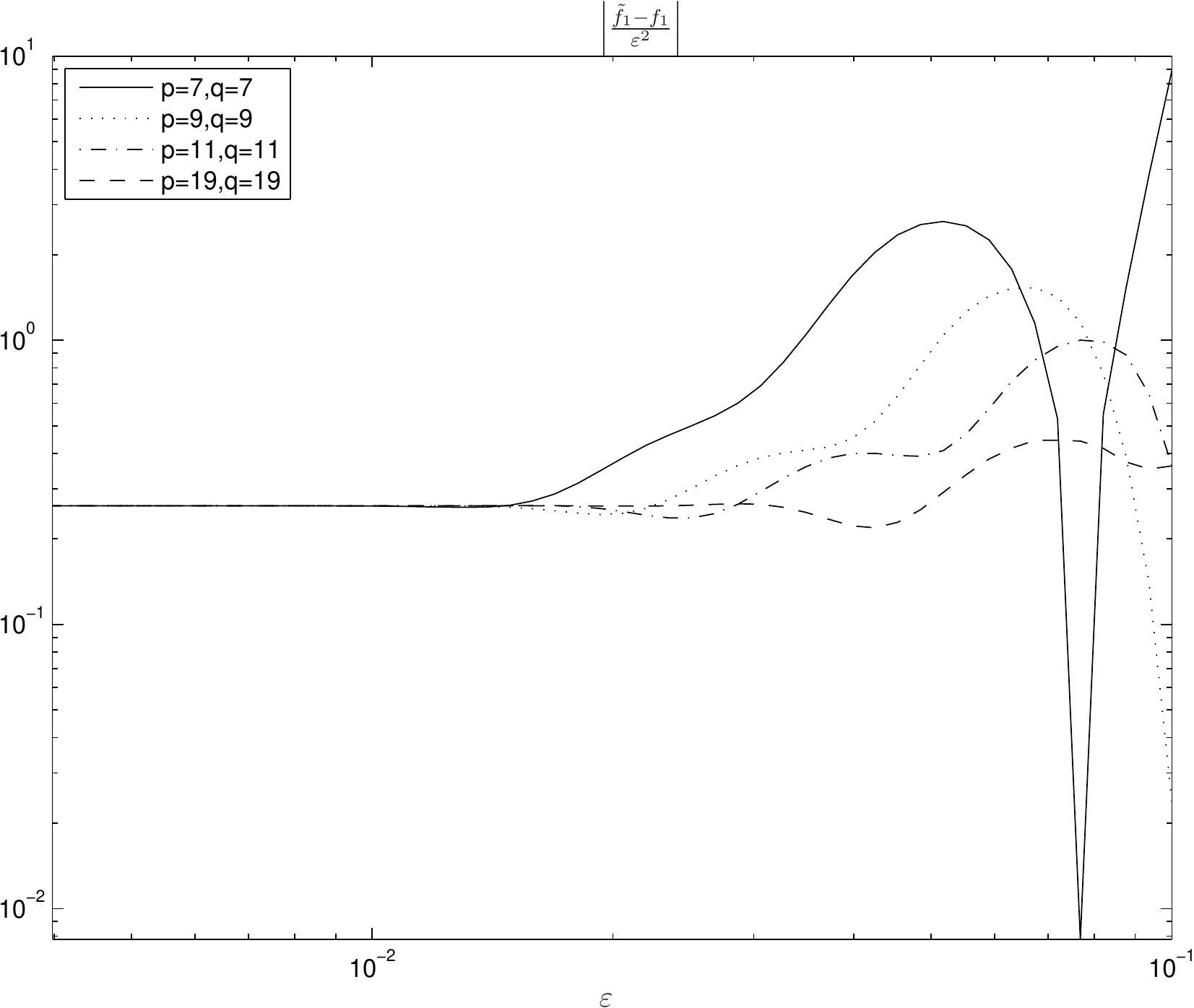}%
	\includegraphics[width=0.5\textwidth]{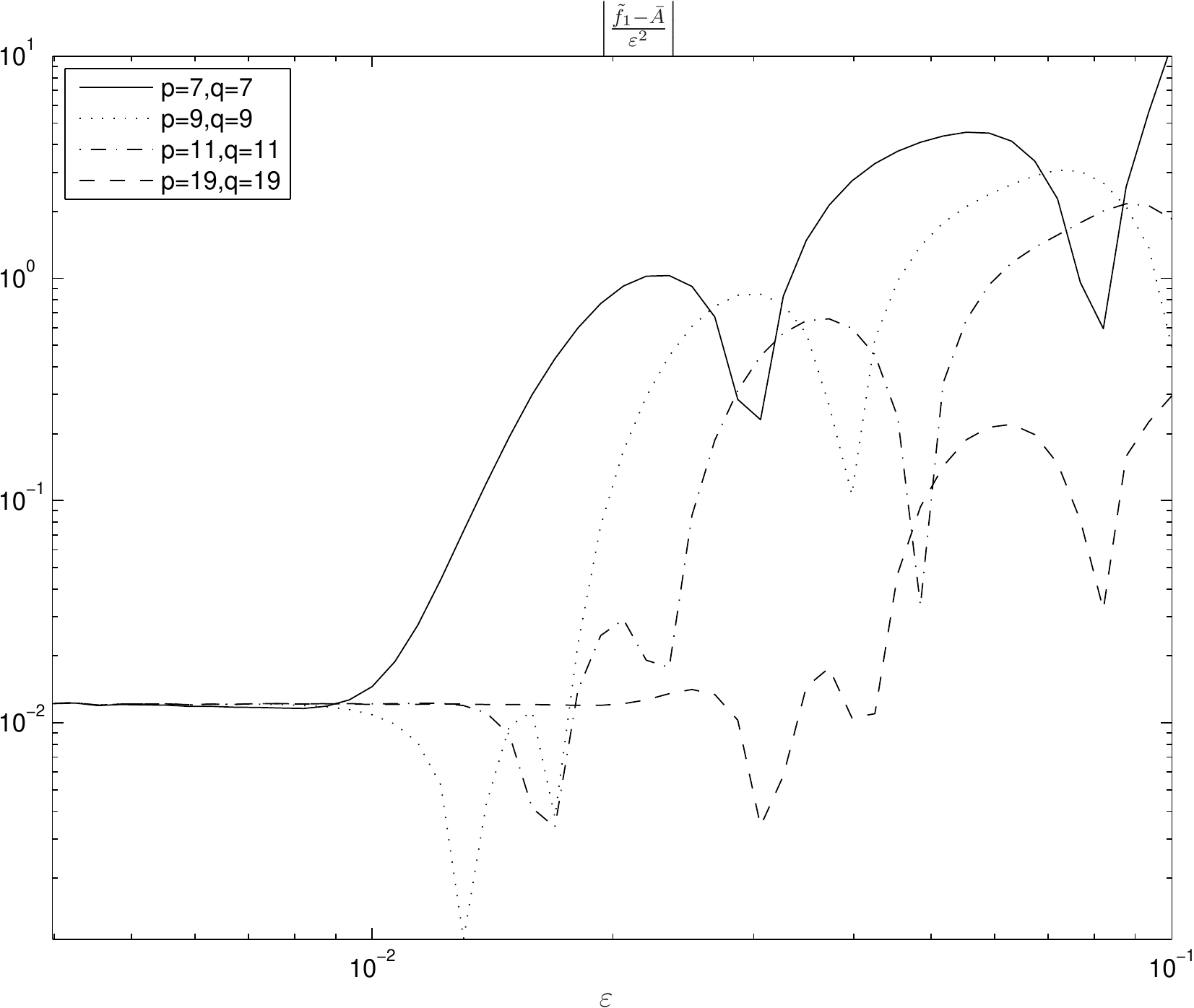}%
	\captionof{figure}{A log-log plot of the difference $\frac{|\tilde{f}_1-f_1|}{\varepsilon^2}$ (left) and $\frac{|\tilde{f}_1-\bar{A}|}{\varepsilon^2}$ (right) at $x=0.3$ for the material $A^{\varepsilon}_2$, 
	using kernels with various $p$- and $q$-parameters.}
	\label{fig:anotherflux1}
\end{center}
\end{figure}
\begin{figure}[h!t]
\begin{center}
	\includegraphics[width=0.5\textwidth]{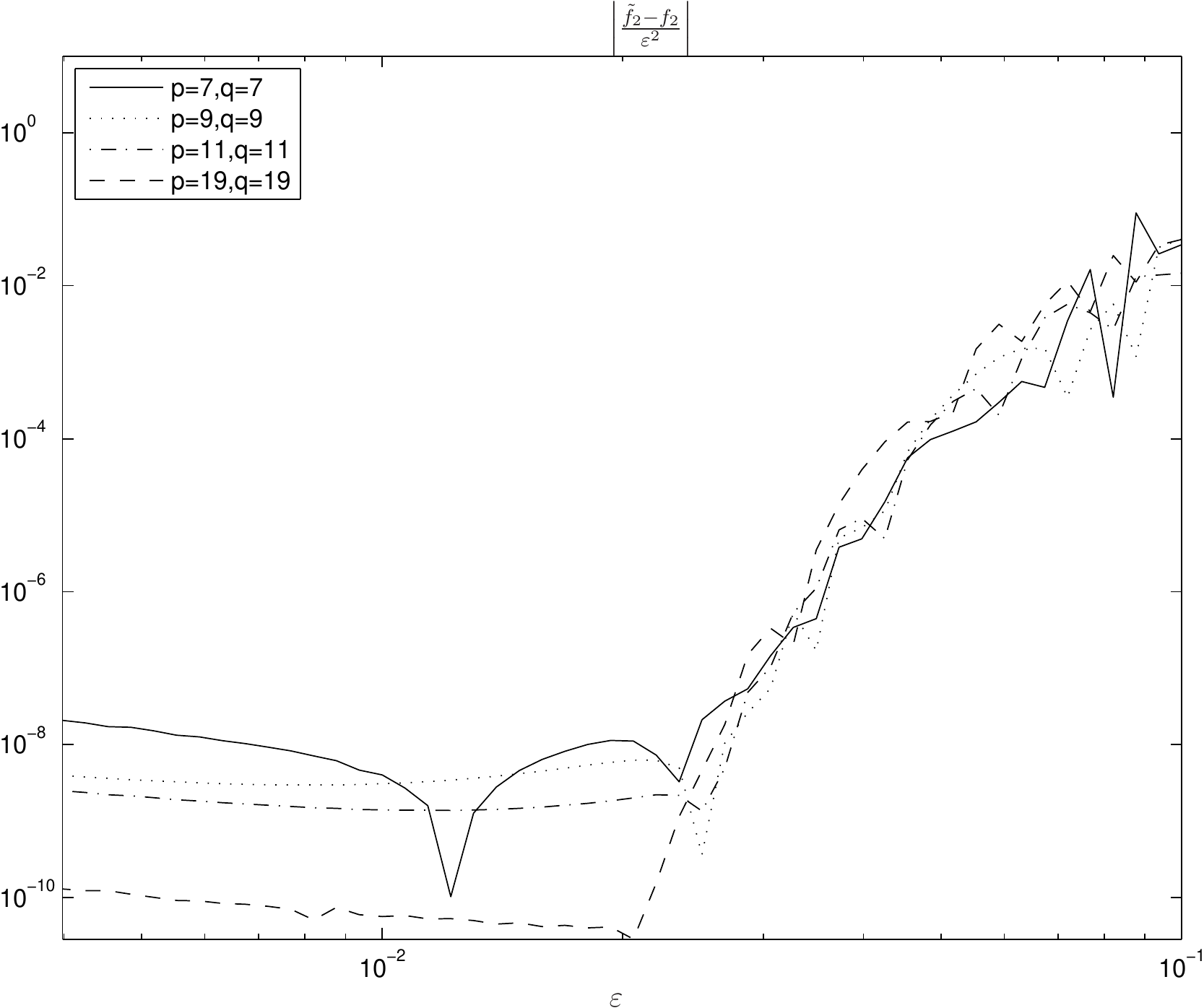}%
	\includegraphics[width=0.5\textwidth]{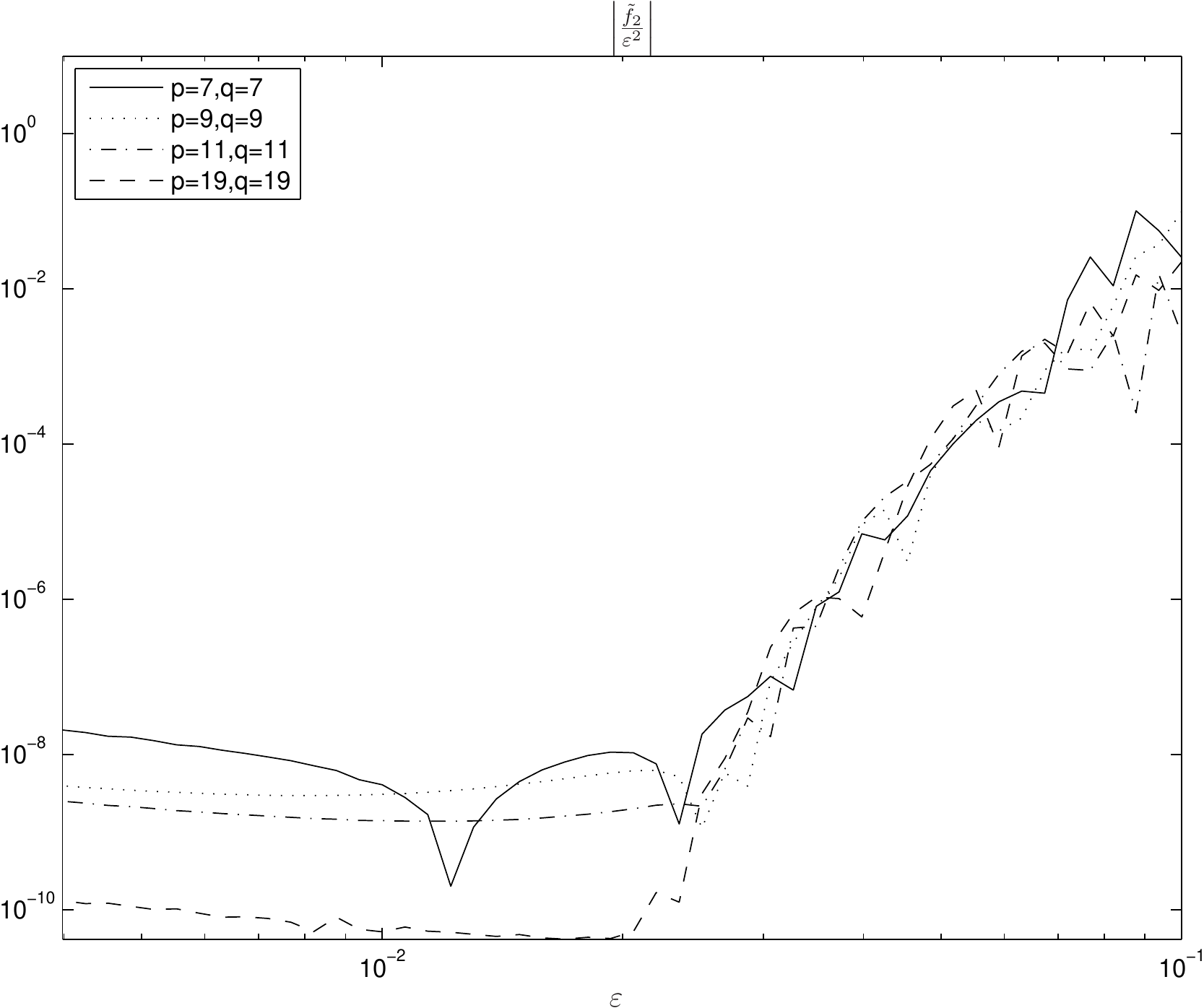}%
	\captionof{figure}{A log-log plot of the difference $\frac{|\tilde{f}_2-f_2|}{\varepsilon^2}$ (left) and $\frac{|\tilde{f}_2|}{\varepsilon^2}$ (right) at $x=0$ for the material $A^{\varepsilon}_2$,
		using kernels with various $p$- and $q$-parameters. Note that the effective equation flux is zero.}
	\label{fig:flux2}
\end{center}
\end{figure}
\begin{figure}[h!t]
\begin{center}
	\includegraphics[width=0.5\textwidth]{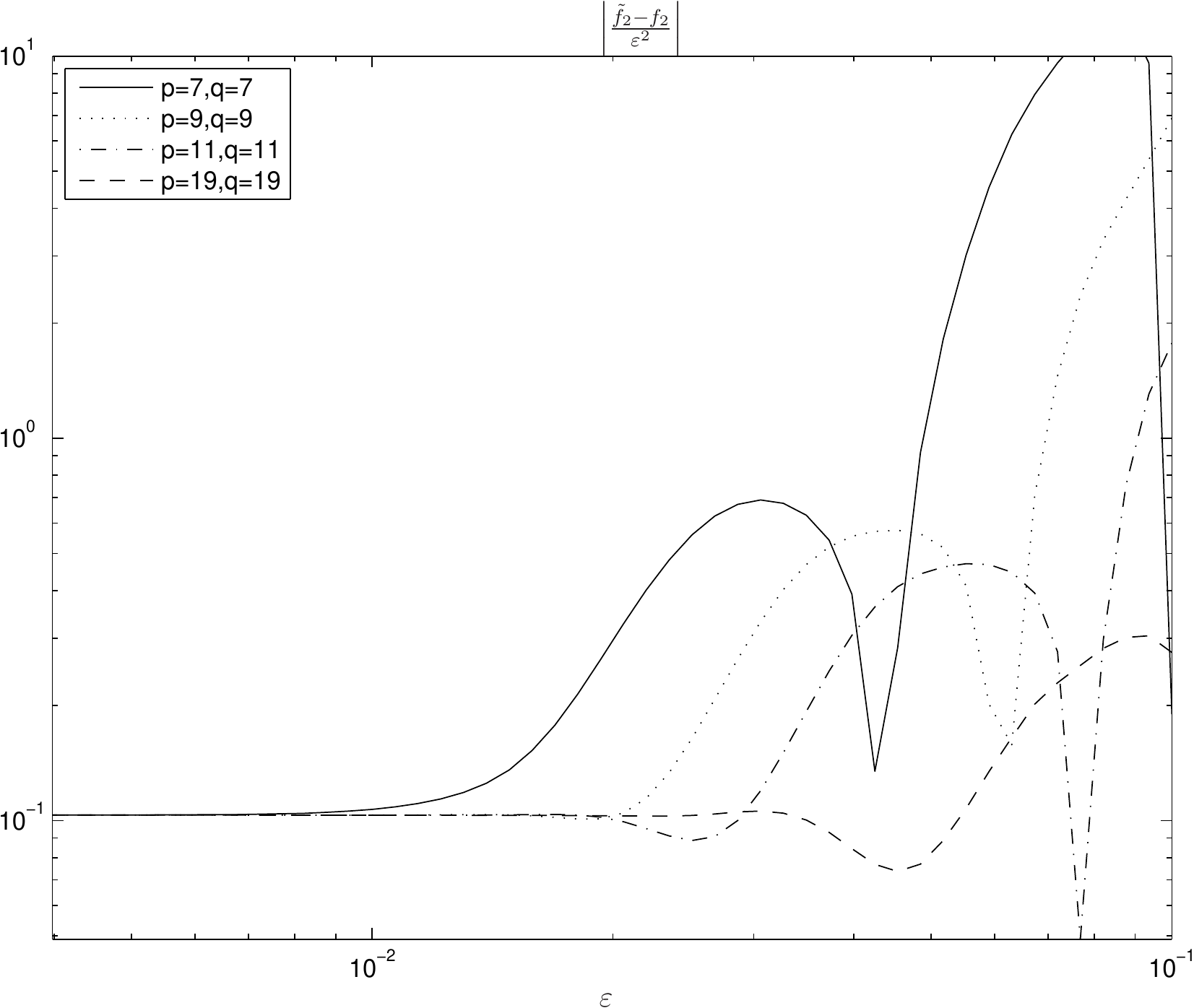}%
	\includegraphics[width=0.5\textwidth]{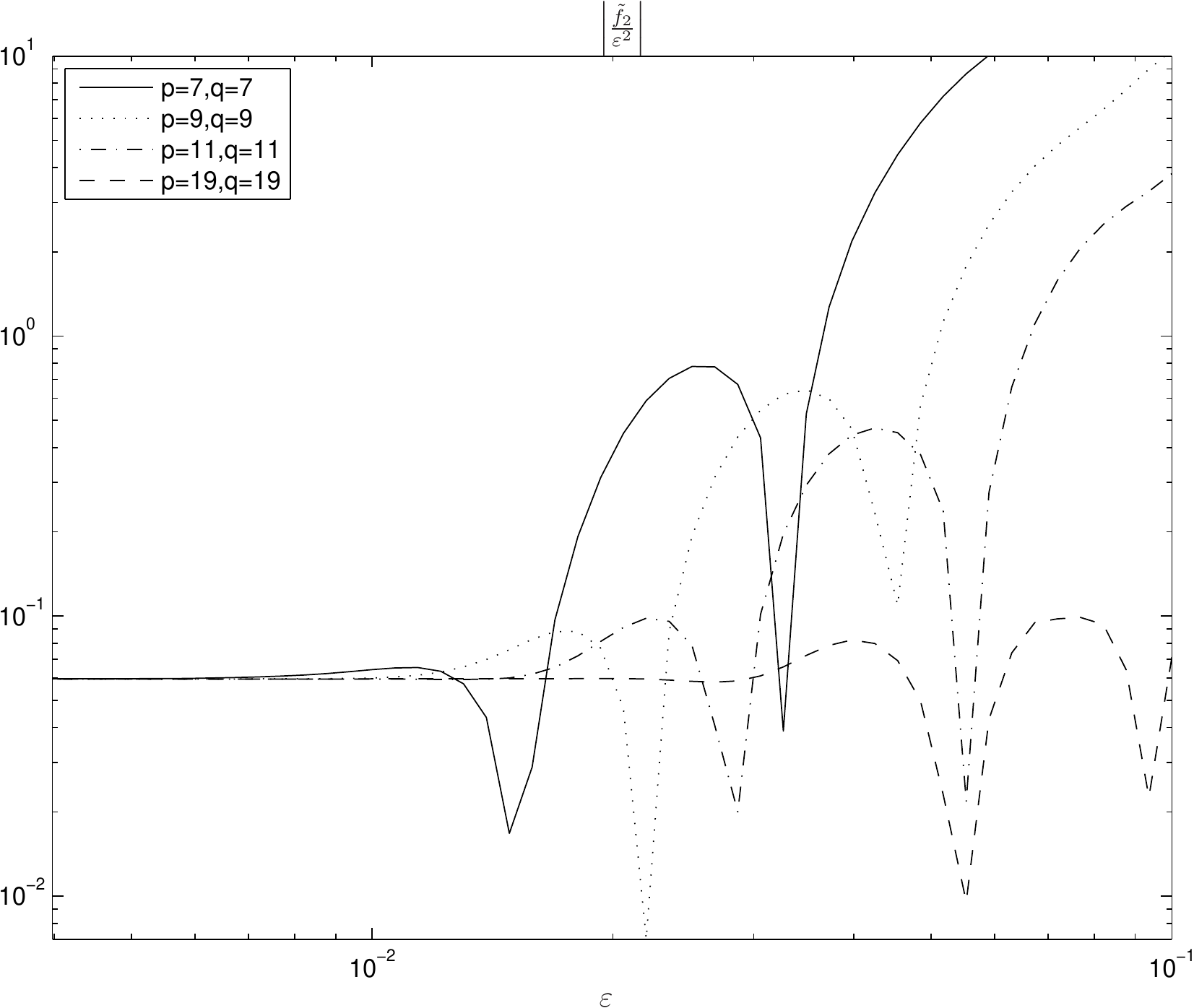}%
	\captionof{figure}{A log-log plot of the difference $\frac{|\tilde{f}_2-f_2|}{\varepsilon^2}$ (left) and $\frac{|\tilde{f}_2|}{\varepsilon^2}$ (right) at $x=0.3$ for the material $A^{\varepsilon}_2$,
		using kernels with various $p$- and $q$-parameters. Note that the effective equation flux is zero.}
	\label{fig:anotherflux2}
\end{center}
\end{figure}
\begin{figure}[h!t]
\begin{center}
	\includegraphics[width=0.5\textwidth]{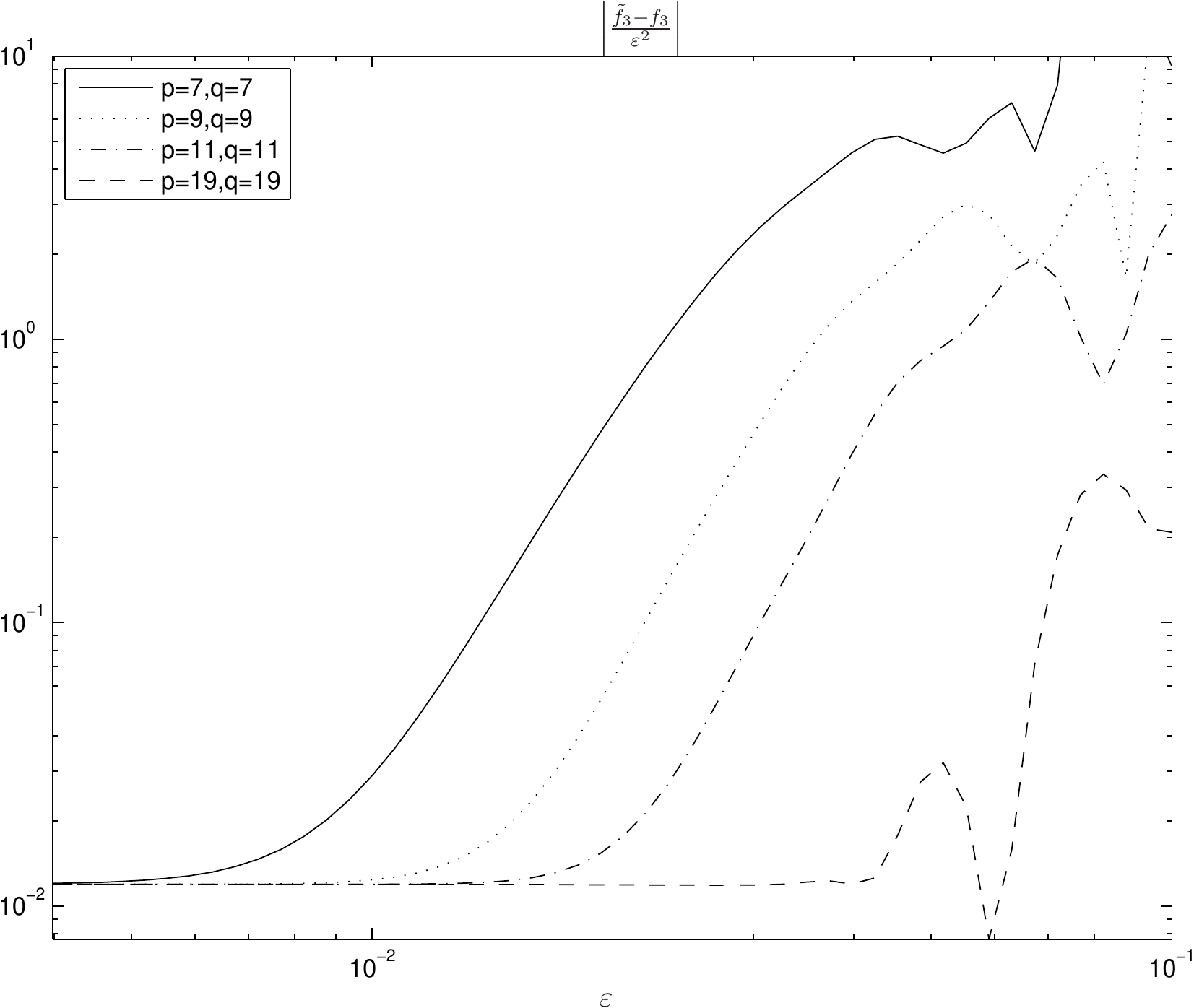}%
	\includegraphics[width=0.5\textwidth]{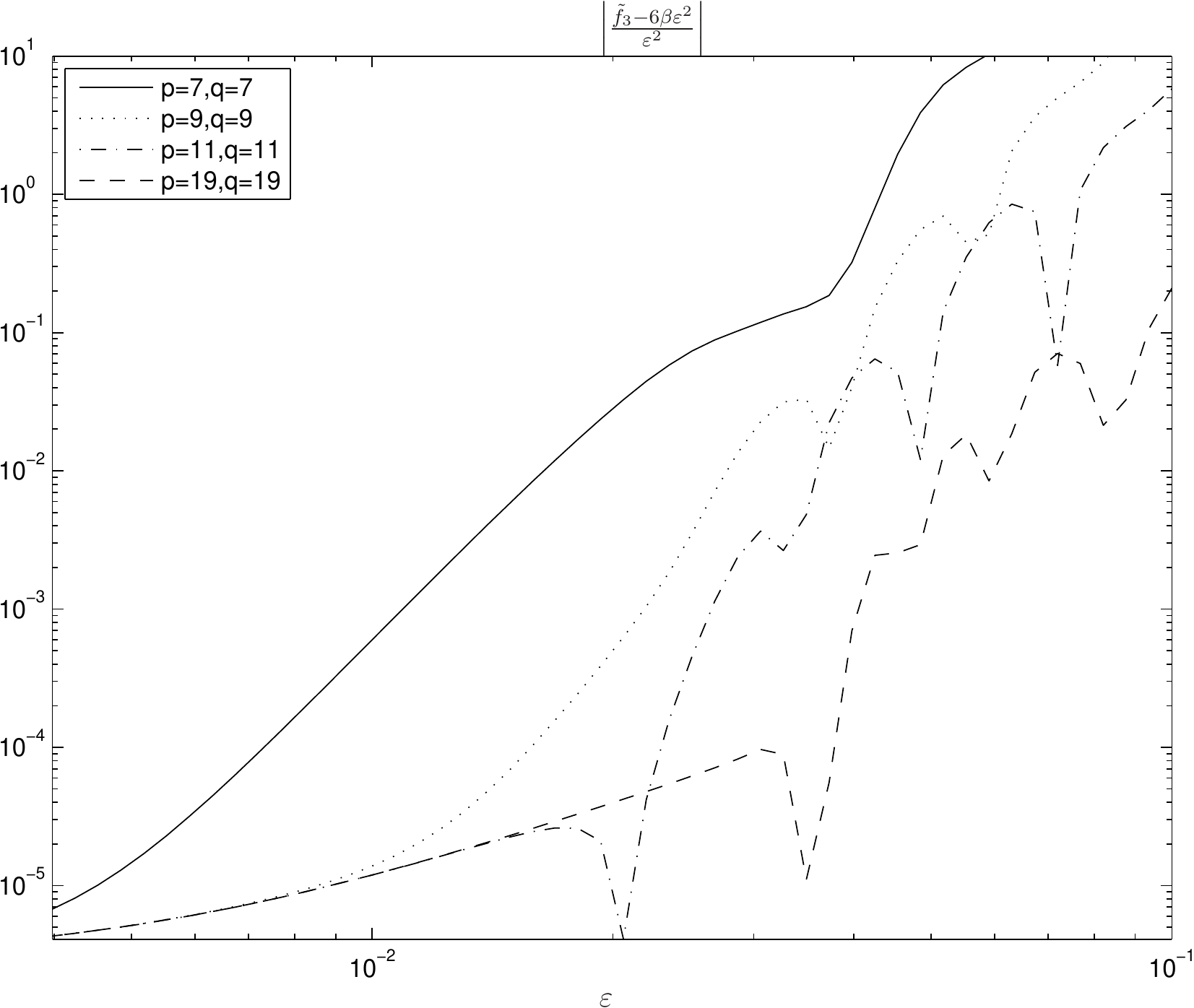}%
	\captionof{figure}{A log-log plot of the difference $\frac{|\tilde{f}_3-f_3|}{\varepsilon^2}$ (left) and $\frac{|\tilde{f}_3-6\beta\varepsilon^2|}{\varepsilon^2}$ (right) at $x=0$ for the material $A^{\varepsilon}_2$,
		using kernels with various $p$- and $q$-parameters.}
	\label{fig:flux3}
\end{center}
\end{figure}
\begin{figure}[h!t]
\begin{center}
	\includegraphics[width=0.5\textwidth]{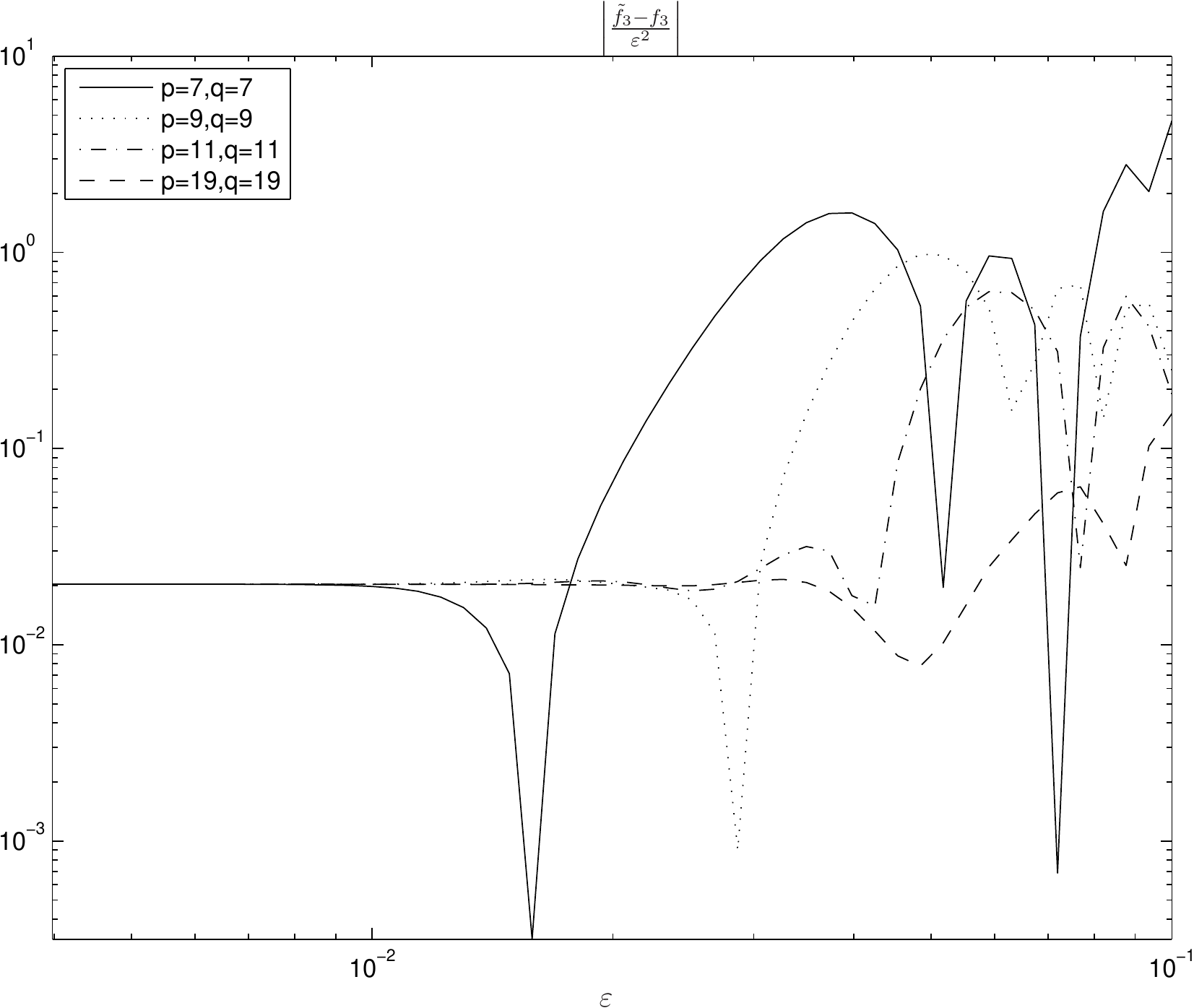}%
	\includegraphics[width=0.5\textwidth]{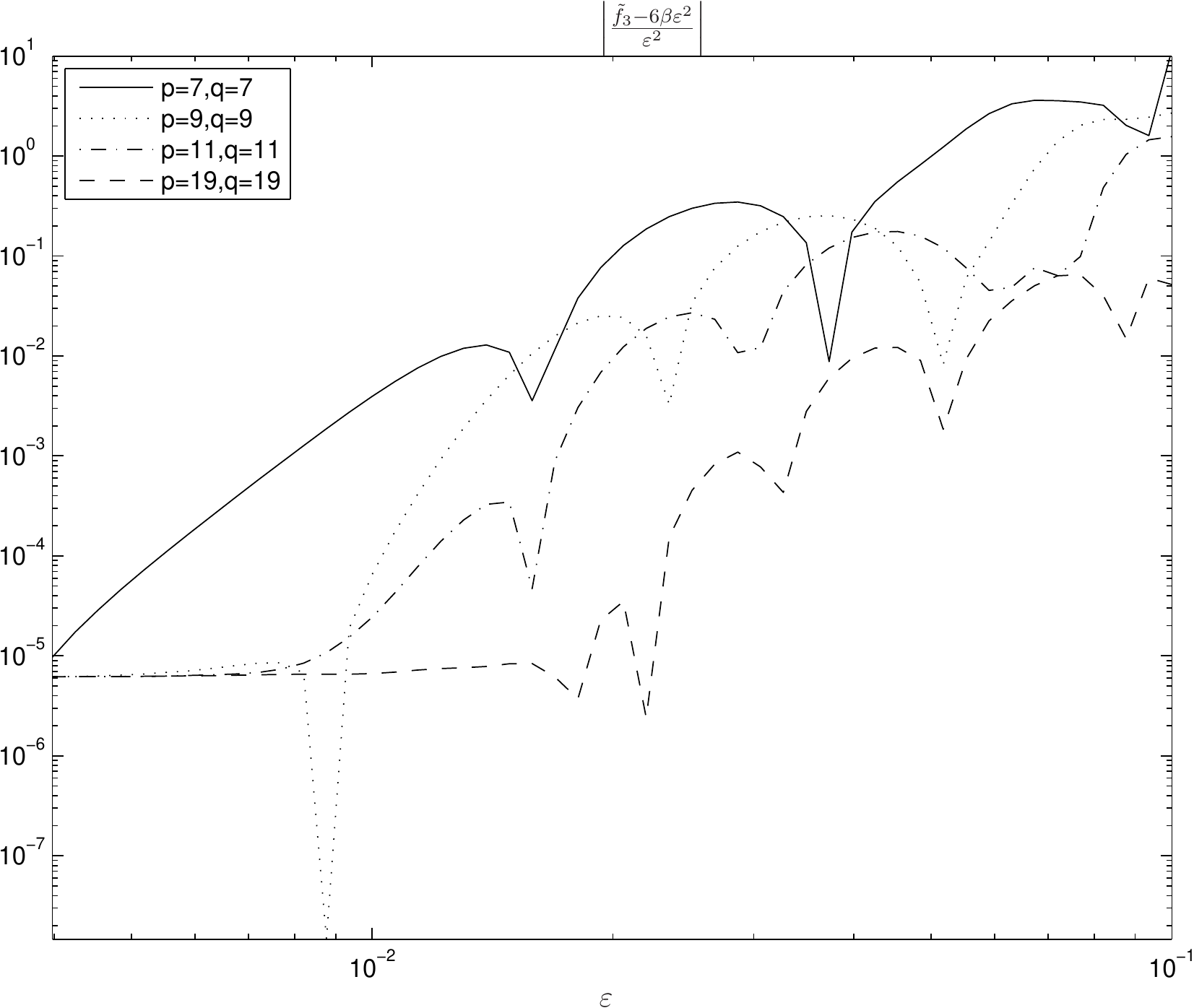}%
	\captionof{figure}{A log-log plot of the difference $\frac{|\tilde{f}_3-f_3|}{\varepsilon^2}$ (left) and $\frac{|\tilde{f}_3-6\beta\varepsilon^2|}{\varepsilon^2}$ (right) at $x=0.3$ for the material $A^{\varepsilon}_2$,
		using kernels with various $p$- and $q$-parameters.}
	\label{fig:anotherflux3}
\end{center}
\end{figure}

We note that as $\varepsilon\to 0$, compared to the
effective flux there is an $\varepsilon^2$-difference in the
components $\tilde{f}_1$ and, when $x=0.3$, also to $\tilde{f}_2$.
There is no difference in $\tilde{f}_3$ and, when $x=0$, no difference in $\tilde{f}_2$.
This gives
us a clue to the form of the macroscopic flux in the case of a
slowly variable coefficients $A(x,x/\varepsilon)$. 
We consider three possible candidates,
\begin{align*}
	F_{\rm alt 1} &= \bar{A}(x) u_{x} + \varepsilon^2 \beta(x) u_{xxx},\\
	F_{\rm alt 2} &= \bar{A}(x) u_{x} + \varepsilon^2 \partial_x\beta(x) u_{xx},\\
	F_{\rm alt 3} &= \bar{A}(x) u_{x} + \varepsilon^2 \partial_{xx}\beta(x) u_{x},
\end{align*}
where $F_{\rm alt 1}$ is the effective equation used above. When
$u=x-x'$ where $x'$ is the macroscopic point of evaluation we get
\begin{align*}
	F_{\rm alt 1}(x') &= \bar{A}(x'),\\
	F_{\rm alt 2}(x') &= \bar{A}(x'),\\
	F_{\rm alt 3}(x') &= \bar{A}(x') + \varepsilon^2\beta''(x').
\end{align*}
Since we have a difference between $\tilde{f}_1$ and $F_{\rm alt 1}(x')=\bar{A}(x')$
of size $\varepsilon^2$ it would indicate that the last alternative
$F_{\rm alt 3}$ is correct. Moreover, when $u=(x-x')^2$, we get
\begin{align*}
	F_{\rm alt 1}(x') &= 0,\\
	F_{\rm alt 2}(x') &= 2\varepsilon^2\beta'(x'),\\
	F_{\rm alt 3}(x') &= 4\varepsilon^2\beta'(x').
\end{align*}
The $\varepsilon^2$-difference in $\tilde{f}_2$ at $x'=0.3$ but not
at $x'=0$ is consistent with both $F_{\rm alt 2}$ and $F_{\rm alt 3}$,
since $\beta'(0)=0$ but $\beta'(0.3)\neq 0$, cf. Fig 
\ref{fig:material2}. Finally, when $u=(x-x')^3$
then all three alternatives have the same value of $6\varepsilon^2\beta(x')$
and the zero difference in $\tilde{f}_3$
is thus consistent with all of them.

These results then lead us to tentatively conjecture the form
of the flux $F$ for the effective equation when the oscillatory coefficient
has a slowly varying part:
$$
	F = \bar{A}(x)u_{x} + \varepsilon^2 \partial_{xx}\beta(x) u_{x}.
$$
This would be the flux that our HMM procedure approximates. It differs
from the effective equation (\ref{eq:num:eff}). Still the numerical
results in Fig. \ref{fig:sol2} shows that HMM and the effective
equation are very close in this case. The $\varepsilon^2$-differences in 
the $\tilde{f}_1$
and $\tilde{f}_2$ components seem to have very little effect on the solution
over the time interval considered. On the other hand, the addition
of the $\tilde{f}_3$-component, which is also of size $\varepsilon^2$, 
has a dramatic effect. Not adding it corresponds to solving
the homogenized equation (\ref{eq:num:hom}), which is far off
the other solutions. We conjecture that for simulations over longer time the
the differences in the $\tilde{f}_1$
and $\tilde{f}_2$ components will play a bigger role.


%
%

\section{Conclusions} \label{section:conclusions}

We have developed and analyzed numerical methods for multiscale wave equations with oscillatory coefficients. The methods are based on the framework of the heterogeneous multiscale method (HMM) and have substantially lower computational complexity than standard discretization algorithms. Convergence is proven in \cite{engquist2011} for finite time approximation in the case of periodic coefficients and for multiple dimensions. The effective equation for long time is different from the finite time homogenized equation. After long time, dispersive effects enter and the method has to capture additional effects on the order of $\mathcal{O}(\varepsilon^2)$~\cite{santosa1991}. Numerical experiments show that the new techniques both accurately and efficiently captures the macroscopic behavior for both finite and long time. It is emphasized that the HMM approach with just minor modifications accurately captures these dispersive phenomena. We prove that our method is stable if the spatial grid in the macro solver is sufficiently coarse.


\bibliographystyle{plain}
\bibliography{HMMLTP}


\end{document}